\documentclass[10.5 pt, a4paper]{amsart}
\usepackage{amssymb,latexsym,amsmath,amsfonts,amsthm}

\usepackage{tikz}
\usetikzlibrary{arrows}
\usetikzlibrary{decorations.markings}

\newcommand{\cev}[1]{\reflectbox{\ensuremath{\vec{\reflectbox{\ensuremath{#1}}}}}}

\numberwithin{equation}{section}

\usepackage[mathscr]{eucal}

\usepackage[T1]{fontenc}
\usepackage{txfonts}
\usepackage[pdftex,bookmarks=true]{hyperref}
\usepackage{graphicx}
\usepackage{enumerate}

\usepackage[all]{xy}

\usepackage{enumerate}

\theoremstyle{plain}
\newtheorem{theorem}{Theorem}[section]

\newtheorem{corollary}[theorem]{Corollary}
\newtheorem{lemma}[theorem]{Lemma}
\newtheorem{prop}[theorem]{Proposition}

\theoremstyle{remark}

\newtheorem{exa}[theorem]{Example}

\theoremstyle{definition}
\newtheorem{dfn}[theorem]{Definition}

\newtheorem{case}{Case}

\DeclareMathOperator{\U}{U}
\DeclareMathOperator{\Star}{Star}

\DeclareMathOperator{\Deg}{deg}
\author{Shiv Parsad}
\address{
Indian Institute of Science Education and Research (IISER) Bhopal\\
Bypass Road, Bhauri\\
Bhopal 462 066\\
Madhya Pradesh, India} 
\email{parsad.shiv@gmail.com}
\author{Bidyut Sanki}
\address{
Institute of Mathematical Sciences\\ % \hfill (Received 00 00 2010)\\
CIT Campus, Tharamani  \\ %\hfill (Revised  00 00 2010)\\
Chennai, 600113\\
India}
\email{bidyut.iitk7@gmail.com}

\begin{document}
\title{Filling systems on surfaces}



\maketitle

%%%%%%%%%%%%%%%%%%% ABSTRACT %%%%%%%%%%%%%%%%%%%%%%%%
\begin{abstract}

Let $F_g$ be a closed orientable surface of genus $g$. A set $\Omega = \{ \gamma_1, \dots, \gamma_s\}$ of pairwise non-homotopic simple closed curves on $F_g$ is called a \emph{filling system} or simply a \emph{filling} of $F_g$, if $F_g\setminus \Omega$ is a union of $b$ topological discs for some $b\geq 1$. A filling system is called \emph{minimal}, if $b=1$. The \emph{size} of a filling is defined as the number of its elements. We prove that the maximum size of a filling of $F_g$ with $b$ complementary discs is $2g+b-1$. Next, we show that for $g\geq 2, b\geq 1\text{ with }(g,b)\neq (2,1)$ (resp. $(g,b)=(2,1)$) and for each $2\leq s\leq 2g+b-1$ (resp.  $3\leq s\leq 2g+b-1$), there exists a filling of $F_g$ of size $s$ with $b$ complementary discs. 

Furthermore, we study geometric intersection number of curves in a minimal filling. For $g\geq 2$, we show that for a minimal filling $\Omega$ of size $s$, the \emph{geometric intersection numbers} satisfy $\max \left\lbrace
i(\gamma_i, \gamma_j)| i\neq j\right\rbrace\leq 2g-s+1$, and for each such $s$ there exists a minimal filling $\Omega=\left\lbrace \gamma_1, \dots, \gamma_s \right\rbrace$ such that $\max\left\lbrace i(\gamma_i, \gamma_j) | i\neq j\right\rbrace = 2g-s+1$.
\end{abstract}
\maketitle

%%%%%%%%%%%%%%%%%%% Section 1 (Introduction) %%%%%%%%

\section{Introduction}
In this article, we study filling systems on oriented closed surfaces. A set $\Omega$ of pairwise non-homotopic simple closed curves on a closed oriented surface $F_g$ is called a filling system or simply a filling of $F_g$, if the complement $F_g\setminus \Omega$ is a disjoint union of topological discs.  It is assumed that the curves in a filling are in \emph{minimal position}, i.e., for $\alpha\neq \beta \in \Omega$, the geometric intersection number is $i(\alpha, \beta)=|\alpha\cap\beta|.$

Filling systems of closed surfaces have become increasingly important in the study of the mapping class group of surfaces and the moduli space of hyperbolic surfaces through the systolic function, in particular. The study of filling system has its origin in the work of Thurston~\cite{WT}, where the author has defined a subset $\chi_g$ of moduli space $\mathcal{M}_g$ of a closed orientable surface of genus $g$, consisting of the hyperbolic surfaces whose systolic geodesics form a filling system. In~\cite{WT}, Thurston proposed $\chi_g$ as a candidate for a spine of $\mathcal{M}_g$ and has provided a sketch of a proof of the claim, but the proof is incomplete. In~\cite{TW2}, Thurston has used filling systems of size two to construct pseudo-Anosov~\cite{FM} mapping classes, and later Penner~\cite{RCP} generalize this construction to pair of multicurves which together forms a filling system. More recently, fillings of surfaces has been studied by Schmutz Schaller~\cite{Schmutz}, Aougab, Huang~\cite{TA}, Fanoni, Parlier~\cite{Parlier}, Sanki~\cite{BS} and others. 

To each filling system, we can associate two numbers; these are $s$, the number of curves in the filling system which we call the \emph{size} of a filling and $d$, the number of discs in the complement. Note that, we call a filling system as a \emph{filling pair}, if its size is two,  and a \emph{minimal filling}, if $d=1$.  Euler's equation implies that $\sum i(\gamma_i, \gamma_j)=2g-2+d$, when $\Omega=\left\lbrace \gamma_1, \dots, \gamma_s \right\rbrace$ is a  filling with $d$ complementary discs.
%Fillings of surfaces are studied extensively by Augab and Huang~\cite{TA}, Anderson, Parlier, Pettet~\cite{JA}, Sanki~\cite{BS}.
 
In [AH15], Aougab and Huang have studied \emph{minimal filling pairs}, i.e., the case when $s=2$ and $d=1$.  
%The mapping class group $\Mod(F_g)$ acts on the set of fillings with the same number of components in the complement.  In~\cite{TA} [Theorem 1.1], Aougab and Huang
The authors have shown that the growth of the number of mapping class group orbits of minimal filling pairs is exponential in the genus of the surface (see theorem 1.1~\cite{TA}). The complexity $T_1(\Omega)$ of a filling $\Omega$ is defined as the number of simple closed curves (up to isotopy) intersecting $\bigcup_{\gamma\in \Omega}\gamma$ only once. Then for a filling pair $\Omega$, $T_1(\Omega)\leq 4g-2$ with equality if and only if $\Omega$ is a minimal filling (see Theorem 1.2~\cite{TA}). 

In~\cite{BS} [Theorem 1.4], Sanki has shown that for every pair of integers $(g, d)$ with $(g, d)\neq (2, 1)$, there exists a \emph{filling pair} of $F_g$ such that the complement is a disjoint union of $d$ topological discs. Moreover, there is no minimal filling pair of a closed surface of genus $2$.

In~\cite{JA} [Theorem 1], Anderson, Parlier and Pettet have proved that, a filling system $\Omega$ of size $n$ with pairwise intersecting no more than $k$ times satisfies $k(n^2-n) \geq 4g-2$ and moreover, the inequality is sharp. Further, if $\Omega$ is a size $n$ filling system of \emph{systolic} curves of a closed hyperbolic surface of genus $g$, then $n> \pi \sqrt{g(g-1)}/\log(4g-2)$ and there exist hyperbolic surfaces of genus $g$ with a filling set of $n \;(\leq 2g)$ systolic curves (Theorem 3~\cite{JA}).  

In this paper, we study the admissible sizes of generic filling systems and also study geometric intersection numbers of curves in minimal fillings on oriented closed surfaces $F_g, g\geq 2$. We define, 
\begin{align*}
L_{g,b}&:=\min\{k| \text{ there exists a filling of $F_g$ of size $k$ with $b$ complementary discs}\}\; \text{and}\\\U_{g,b}&:=\max\{k|\text{ there exists a filling of $F_g$ of size $k$ with $b$ complementary discs}\}.
\end{align*}

It follows from Theorem 1.2, Proposition 4.2 in~\cite{BS} and Theorem 1.1 in~\cite{TA} that 
\[L_{g,b}= \begin{cases} 
      3 & \text{ if }(g,b)=(2,1) \text{ and}\\
      2 & \text{ otherwise }. 
   \end{cases}
\]
 Therefore, we focus on the growth of $\U_{g,b}$. We prove the theorem below:

\begin{theorem}\label{res:1}
$\U_{g,b}=2g+b-1$, for all $g,b\in \mathbb{N}$.
\end{theorem}

The proof of Theorem~\ref{res:1} is constructive and the main ingredients are basic (topological) graph theoretic arguments, cellular decomposition of surfaces, and a construction called  \emph{join} of graphs (see subsection~\ref{jn}).
%
%and Euler's equation. %The proof has two steps. In the first step, we show that $\U_g\leq 2g$ and in the final step, we construct a minimal filling of $F_g$ of size $2g$. 

\begin{theorem}\label{res:2}
Let $g\geq 2$ be an integer. For $b\geq 1$ and each $s$ satisfying $L_{g,b}\leq s \leq \U_{g,b}$, there exists a filling $\Omega_{b,s}$ of $F_g$ such that $F_g\setminus \Omega_{b,s}$ is a disjoint union of $b$ discs, and $|\Omega_{b,s}|=s$. 
\end{theorem}

The proof of Theorem~\ref{res:2} is by explicit construction in an inductive procedure. In particular, we use three operations on fat graphs so-called \emph{join}, \emph{connected sum}, and \emph{plumbing} (see subsections~\ref{jn},~\ref{cs}, and~\ref{pl}).

Finally, we study geometric intersection numbers of the curves in minimal fillings of $F_g$. 
\begin{theorem}\label{res:3}
Let $\Omega=\left\lbrace \gamma_1, \dots, \gamma_s \right\rbrace$ be a minimal filling of $F_g, g\geq 2$. Then we have $$\max\left\lbrace i(\gamma_i, \gamma_j)| i\neq j \right\rbrace\leq 2g-s+1.$$ Furthermore, for each $s$ satisfying $L_g\leq s \leq \U_g$, there exists a minimal filling $\Omega_s$ of size $s$ such that $$\max\left\lbrace i(\alpha, \beta)| \alpha \neq \beta \in \Omega_s\right\rbrace = 2g-s+1.$$
\end{theorem}

%%%%%%%%%%%%%%%%%%%%% Section 2: Fat graphs and operations on fat graphs--(1) Join of two fat graphs, (2) Connected sum of two fat graphs and (3) Plumbing of two fat graphs. %%%%

\section{Fat graphs}
In this section, we recall some definitions from graph theory, in particular fat graphs (Section 2 in~\cite{BS}). Next, we define three binary operations on fat graphs so-called \emph{join}, \emph{plumbing}, and \emph{connected sum}. We also develop three technical results namely, Proposition~\ref{prop:join}, Proposition~\ref{connected_sum}, and Proposition~\ref{plumb:boundary}, which are essential in the subsequent sections.

\begin{dfn}
A graph is a triple $G=(E, \sim, \sigma_1)$, where 
\begin{enumerate}
\item $E$ is a set of directed edges containing an even number of elements.

\item  $\sim$ is an equivalence relation on $E$ and 

\item  $\sigma_1:E\to E$ is a fixed point free involution which maps a directed edge to the edge with the reverse direction.
\end{enumerate}
The set $V:=E/\!\!\sim$ of equivalence classes of $\sim$ is the vertex set and $E_1:=E/\sigma_1$ is the set of undirected edges of $G$. 
\end{dfn}
%cyclic order on the set of directed edges at the equivalence classes of \sim

A fat graph structure on $G$ is a permutation $\sigma_0: E \to E$, whose cycles correspond to cyclic order on the set of directed edges emanating from each vertex. We assume that the degree of each vertex of a fat graph is at least $3$. A fat graph is called \emph{decorated} if the degree of each vertex is an even integer. A cycle in a fat graph is called \emph{standard} if every two consecutive edges in the cycle are opposite to each other in the cyclic order at their shared vertex. Given a fat graph $\Gamma$, the set of boundary components is denoted by $\partial\Gamma$, which corresponds to the set of cycles of the permutation $\sigma_1*\sigma_0^{-1}$ (see Lemma 2.4 in~\cite{BS}). 

By thickening the edges of a fat graph $\Gamma$ one obtain a unique topological surface and its genus is called the genus of the fat graph which we denote by $g(\Gamma)$. 

\begin{exa}\label{eg:triple_2}
Consider the  fat graph $G_1=(E,\sim,\sigma_1,\sigma_0)$ of genus $2$ (see Figure~\ref{eg:1}), described as follows:

\begin{enumerate}
\item $E=\{\vec{f}_i, \cev{f}_i\;|\;i=1,2,\dots, 6\}$.
\item The set of equivalence classes of $\sim$ is $V=\{u_1, u_2, u_3\}$, where 
$ u_1=(\vec{f}_1, \vec{f}_2, \vec{f}_3, \vec{f}_4),\\\,u_2=(\cev{f}_3, \cev{f}_4, \vec{f}_5, \cev{f}_2),\text{ and }u_3=(\cev{f}_5, \vec{f}_6, \cev{f}_1, \cev{f}_6).$
\item The fixed point free involution $\sigma_1$ is defined by $\sigma_1(\vec{f}_i)=\cev{f}_i$ for $i=1,2,\dots, 6.$
\item The permutation (fat graph structure) $\sigma_0$ is given by $ \sigma_0=\prod\limits_{i=1}^{3}u_i$.
\end{enumerate}
It is easy to see from Figure~\ref{eg:1} that $G_1$ has one boundary component and $3$ standard cycles of lengths $3,2$ and $1$. Therefore $G_1$ corresponds to a minimal filling triple of $F_2$.
\begin{figure}[htbp]
\begin{center}
\begin{tikzpicture}
\draw [rounded corners=2mm](4, 1) -- (4, 0) -- (2,0) -- (2, 1) -- (0, 1) -- (0, 3.9) -- (0.7, 3.9);

\draw [rounded corners=2mm] (3.7, 1)-- (3.7, 0.3) -- (2.3, 0.3) -- (2.3, 1) -- (5, 1) -- (5, 3.9) -- (4,3.9) -- (4, 5.6) -- (0.7,5.6) -- (0.7, 2.6) -- (2,2.6) -- (2, 3.6) -- (3.7, 3.6) -- (3.7, 2.9) -- (3, 2.9) -- (3, 3.6);

\draw [rounded corners=2mm](4, 1.3) -- (4, 2.3) -- (2,2.3) -- (2, 1.3) -- (0.3, 1.3) -- (0.3, 3.6) -- (0.7, 3.6);

\draw [rounded corners=2mm](3.7, 1.3) -- (3.7, 2) -- (2.3,2) -- (2.3, 1.3) -- (4.7, 1.3) -- (4.7, 3.6) -- (4, 3.6) -- (4, 2.6) -- (2.7, 2.6) -- (2.7, 3.6);

\draw [rounded corners=2mm](3, 3.9) -- (3, 4.9) -- (1.7,4.9) -- (1.7, 3.9) -- (1, 3.9);

\draw [rounded corners=2mm](2.7, 3.9) -- (2.7, 4.6) -- (2,4.6) -- (2, 3.9) -- (3.7, 3.9) -- (3.7, 5.3) -- (1, 5.3)-- (1, 2.9) -- (1.7,2.9) -- (1.7, 3.6) -- (1, 3.6); 

\draw [->] (4.5, 3.6) -- (4.51, 3.6)node [below] {$\vec{f}_1$}; \draw [->] (2.5, 3.6) -- (2.51, 3.6)node [below] {$\cev{f}_3$}; \draw [->] (3.0, 1) -- (3.01, 1)node [below] {$\cev{f}_1$};  \draw [->] (1.5, 1.3) -- (1.49, 1.3)node [above] {$\cev{f}_5$}; \draw [->] (1.25, 3.9) -- (1.24, 3.9)node [above] {$\cev{f}_5$}; \draw [->] (3.21, 3.9) -- (3.2, 3.9)node [above] {$\vec{f}_3$}; \draw [->] (2.3, 1.7) -- (2.3, 1.71)node [right] {$\cev{f}_6$}; \draw [->] (4, 4.4) -- (4, 4.41)node [right] {$\vec{f}_2$}; \draw [->] (2, 4.3) -- (2, 4.31)node [right] {$\cev{f}_4$}; \draw [->] (1.7, 3.21) -- (1.7, 3.2)node [left] {$\cev{f}_2$}; \draw [->] (3.7, 3.21) -- (3.7, 3.2)node [left] {$\vec{f}_4$};\draw [->] (2, 0.51) -- (2, 0.5)node [left] {$\cev{f}_6$};
\draw (3.85, 3.75) node {$u_1$};\draw (1.85, 3.75) node {$u_2$}; \draw (2.15, 1.15) node {$u_3$};

\end{tikzpicture}
\end{center}
\caption{$G_1$}\label{eg:1}
\end{figure}
\end{exa}

%%%%%%%%%%%%%%%%%%%%   Join of two fat graphs %%%%%%%%%%%%%%%%%%%%%%%%%%%

\subsection{Join of two fat graphs}\label{jn}
We define a binary operation on fat graphs called \emph{join}. Let $\Gamma_i, i=1,2$, be two fat graphs. Consider $x=\left\lbrace \vec{x}, \cev{x} \right\rbrace$ and $y =\left\lbrace \vec{y}, \cev{y} \right\rbrace$ are two edges of $\Gamma_1$ and $\Gamma_2$ respectively. In this operation, we cut the edges $x$ and $y$ into two edges $x_i=\left\lbrace \vec{x}_i, \cev{x}_i \right\rbrace$ and $y_i=\left\lbrace \vec{y}_i, \cev{y}_i \right\rbrace$, $i=1, 2$. Then we join $x_1$ with $y_1$ and $x_2$ with $y_2$ at the end points and get new edges $e$ and $f$. More precisely, $e=\left\lbrace \vec{e}, \cev{e} \right\rbrace$ and $f=\left\lbrace \vec{f}, \cev{f} \right\rbrace$ where,
\begin{align*}
\vec{e} = \vec{x}_1*\cev{y}_1 \text{ and } \cev{e} = \vec{y}_1*\cev{x}_1,\\ \vec{f} = \cev{y}_2 * \vec{x}_2 \text{ and } \cev{f} = \cev{x}_2 * \vec{y}_2.
\end{align*}
Here, $*$ is concatenation. For a local picture of the operation, we refer to Figure~\ref{join}. Now, we describe the new graph $\Gamma = \left(\Gamma_1, x\right)\# \left(\Gamma, y\right)$, call the join of $\Gamma_1$ and $\Gamma_2$ along $x$ and $y$, below: $\Gamma=\left( E, \sim, \sigma_1, \sigma_0\right)$
\begin{enumerate}
\item The set of directed edges is given by, $$E = \left( E_1 \cup E_2 \setminus \left\lbrace \vec{x}, \cev{x}, \vec{y}, \cev{y} \right\rbrace \right) \cup \left\lbrace \vec{e}, \cev{e}, \vec{f}, \cev{f} \right\rbrace,$$ where $E_1$ and $E_2$ are the set of directed edges of $\Gamma_1$ and $\Gamma_2$ respectively.   

\item The fixed point free involution $\sigma_1$ is defined as usual $\sigma\left( \vec{a} \right) = \cev{a}$ for all $\vec{a}\in E$.

\item Let $v = \left\lbrace \vec{e}_i: i=1, \dots, m \right\rbrace\in V_1\cup V_2$, where $V_1$ and $V_2$ are the set of vertices $\Gamma_1$ and $\Gamma_2$ respectively. We define $v'=\left\lbrace {\vec{e}_i}^{\,'} : i=1,\dots, m\right\rbrace$, where
\[
{\vec{e}_i}^{\,'}=
\begin{cases}
\vec{e}_i & \text{ if } \vec{e}_i\in E_1 \cup E_2 \setminus \left\lbrace \vec{x}, \cev{x}, \vec{y}, \cev{y} \right\rbrace \text{ and }\\
\vec{e}, \cev{e}, \vec{f} \text{ and } \cev{f} & \text{if } \vec{e}_i=\vec{x}, \vec{y}, \cev{y} \text{ and } \cev{x} \text{ respectively.}
\end{cases}
\]
Then the set equivalence classes of the relation $\sim$ is $V=E/\sim = \left\lbrace v': v\in V_1\cup V_2\right\rbrace.$
\item If $\sigma_v = \left( \vec{e}_1, \dots , \vec{e}_m\right)$ is the cyclic order at the vertex $v\in V_1\cup V_2$, then we define $\sigma_{v'} = \left( {\vec{e}_1}^{\,'}, \dots, {\vec{e}_m}^{\,'}\right)$ is the cyclic order at the vertex $v'\in V$. Thus $\sigma_0=\prod\limits_{v'\in V}\sigma_{v'}.$
\end{enumerate}

\begin{prop}\label{prop:join}
Let $\Gamma_i, i=1, 2,$ be two decorated fat graphs of genus $g_i$ with $b_i$ boundary components and $s_i$ standard cycles. Consider $x=\left\lbrace \vec{x}, \cev{x} \right\rbrace$ and $y=\left\lbrace \vec{y}, \cev{y} \right\rbrace$ are two edges of $\Gamma_1$ and $\Gamma_2$ respectively. Then the join $\Gamma = \left( \Gamma_1, x \right) \# \left( \Gamma_2, y\right)$ of $\Gamma_1$ and $\Gamma_2$ along $x, y$ has the following properties:
\begin{enumerate}
\item The number of boundary components in $\Gamma$ is given by 
\[
\lvert \partial\Gamma \rvert=
\begin{cases}
b_1+b_2 & \text{ if } \vec{x}, \cev{x} \in \partial \text{ and } \vec{y}, \cev{y}\in \eta  \text{ for some $\partial\in \partial\Gamma_1, \eta\in \partial\Gamma_2$, and}\\
b_1+b_2-2 & \text{ otherwise. } 
\end{cases}
\]

\item The number of standard cycles in $\Gamma$ is $s_1+s_2-1$.

\item The genus of $\Gamma$ is given by
\[
g\left( \Gamma \right)=
\begin{cases}
g_1+g_2-1 & \text{ if } \vec{x}, \cev{x} \in \partial \text{ and } \vec{y}, \cev{y}\in \eta  \text{ for some $\partial\in \partial\Gamma_1, \eta\in \partial\Gamma_2$, and }\\
g_1+g_2 & \text{ otherwise. } 
\end{cases}
\]
\end{enumerate}
\end{prop}

\begin{proof}(1) 
Let us assume that $\partial\Gamma_1=\left\lbrace \partial_1, \dots, \partial_{b_1} \right\rbrace$ and $\partial\Gamma_2=\left\lbrace \zeta_1, \dots, \zeta_{b_2} \right\rbrace$. Now, we consider the following four cases:
\begin{case}
In this case, we assume that $\vec{x}, \cev{x}\in \partial_1$ and $\vec{y}, \cev{y}\in \zeta_1$ for some $\partial_1\in \partial\Gamma_1,\text{ and }\zeta_1\in \partial\Gamma_2$. Thus there are directed paths $P, Q$ and $R, S$ on the boundary of $\Gamma_1$ and $\Gamma_2$ respectively, such that 
\begin{align*}
\partial_1 = P\vec{x}Q\cev{x}=P\vec{x}_1\vec{x}_2Q\cev{x}_2\cev{x}_1 \text{ and } \zeta_1 = R \vec{y} S \cev{y}= R \vec{y}_1\vec{y}_2S\cev{y}_2\cev{y}_1.
\end{align*}
Note that, we split $x$ into $x_1, x_2$ and $y$ into $y_1, y_2$. The set of boundary components of the join fat graph $\Gamma$ is given by (see Figure~\ref{join}) 
\begin{align*}
\partial\Gamma = \left\lbrace \partial_2, \dots, \partial_{b_1} \right\rbrace \cup \left\lbrace \zeta_2, \dots, \zeta_{b_2} \right\rbrace \cup \left\lbrace \eta_1, \eta_2  \right\rbrace,
\end{align*}
where $\eta_1, \eta_2$ are given by following (see Figure~\ref{join}):
\begin{align*}
\eta_1 = P \left( \vec{x}_1 * \cev{y}_1 \right) R \left( \vec{y}_1 * \cev{x}_1 \right) = P\;\vec{e}\;R\;\cev{e} \text{ and } \eta_2 = Q \left( \cev{x}_2 * \vec{y}_2 \right) S \left( \cev{y}_2 * \vec{x}_2 \right)= Q\;\cev{f}\;S\;\vec{f}. 
\end{align*}

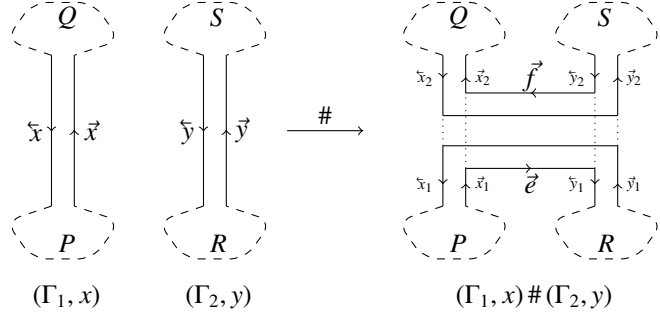
\begin{figure}[htbp]
\begin{center}
\begin{tikzpicture}
\draw (-3.3, -1) -- (-3.3, 1); \draw (-3, -1) -- (-3, 1); \draw [dashed, rounded corners=2mm] (-3.3, -1) -- (-3.9, -1.2) -- (-3.55, -1.7) -- (-2.75,-1.7) -- (-2.4,-1.2) -- (-3, -1); \draw [dashed, rounded corners = 2mm] (-3.3, 1) -- (-3.9, 1.2) -- (-3.55, 1.7) -- (-2.75, 1.7) -- (-2.4, 1.2) -- (-3, 1); \draw [->] (-3.3, 0) -- (-3.3, -0.01)node [left] {$\cev{x}$}; \draw [->] (-3, 0) -- (-3, 0.01)node [right] {$\vec{x}$}; \draw (-3.1, 1.5) node {$Q$}; \draw (-3.1, -1.5) node {$P$}; 
\draw (-3.1, -2.2) node {$\left( \Gamma_1, x \right)$}; \draw (-1.1, -2.2) node {$\left( \Gamma_2, y \right)$};

\draw (-1.3, -1) -- (-1.3, 1); \draw (-1, -1) -- (-1, 1); \draw [dashed, rounded corners=2mm] (-1.3, -1) -- (-1.9, -1.2) -- (-1.55, -1.7) -- (-0.75,-1.7) -- (-0.4,-1.2) -- (-1, -1); \draw [dashed, rounded corners = 2mm] (-1.3, 1) -- (-1.9, 1.2) -- (-1.55, 1.7) -- (-0.75, 1.7) -- (-0.4, 1.2) -- (-1, 1); \draw [->] (-1.3, 0) -- (-1.3, -0.01)node [left] {$\cev{y}$}; \draw [->] (-1, 0) -- (-1, 0.01)node [right] {$\vec{y}$}; \draw (-1.1, 1.5) node {$S$}; \draw (2.05, 1.5) node {$Q$}; \draw (-1.1, -1.5) node {$R$}; \draw (2.05, -1.5) node {$P$}; \draw (4, -1.5) node {$R$}; \draw (4, 1.5) node {$S$};

\draw [->] (-0.2, 0) -- (0.8, 0); \draw (0.3, 0.2) node {$\#$};

\draw (1.85, -1) -- (1.85, -0.2); \draw (1.85, 0.2)-- (1.85, 1); \draw [->] (1.85, -0.7) -- (1.85, -0.71) node [left] {\tiny $\cev{x}_1$}; \draw [->] (2.15, -0.71) -- (2.15, -0.7) node [right] {\tiny $\vec{x}_1$}; \draw [->] (1.85, 0.71) -- (1.85, 0.7) node [left] {\tiny $\cev{x}_2$}; \draw [->] (2.15, 0.7) -- (2.15, 0.71) node [right] {\tiny $\vec{x}_2$};  \draw (2.15, -1) -- (2.15, -0.5); \draw (2.15, 0.5)-- (2.15, 1); \draw [dashed, rounded corners=2mm] (1.85, -1) -- (1.25, -1.2) -- (1.6, -1.7) -- (2.4,-1.7) -- (2.75,-1.2) -- (2.15, -1); \draw [dashed, rounded corners = 2mm] (-1.3+3.15, 1) -- (1.25, 1.2) -- (1.6, 1.7) -- (2.4, 1.7) -- (2.75, 1.2) -- (2.15, 1);

\draw (3.85, -1) -- (3.85, -0.5); \draw (3.85, 0.5)-- (3.85, 1); \draw [->] (3.85, -0.7) -- (3.85, -0.71) node [left] {\tiny $\cev{y}_1$}; \draw [->] (4.15, -0.71) -- (4.15, -0.7) node [right] {\tiny $\vec{y}_1$}; \draw [->] (3.85, 0.71) -- (3.85, 0.7) node [left] {\tiny $\cev{y}_2$}; \draw [->] (4.15, 0.7) -- (4.15, 0.71) node [right] {\tiny $\vec{y}_2$};  \draw (4.15, -1) -- (4.15, -0.2); \draw (4.15, 0.2)-- (4.15, 1); \draw [dashed, rounded corners=2mm] (3.85, -1) -- (3.25, -1.2) -- (3.6, -1.7) -- (4.4,-1.7) -- (4.75,-1.2) -- (4.15, -1); \draw [dashed, rounded corners = 2mm] (-1.3+5.15, 1) -- (3.25, 1.2) -- (3.6, 1.7) -- (4.4, 1.7) -- (4.75, 1.2) -- (4.15, 1); 

\draw (4.15, -0.2) -- (1.85,-0.2); \draw (2.15, -0.5) -- (3.85, -0.5); \draw (3, -0.7) node {$\vec{e}$}; \draw [->] (3, -0.5) -- (3.01, -0.5);
\draw (4.15, 0.2) -- (1.85,0.2); \draw (2.15, 0.5) -- (3.85, 0.5); \draw [->] (3.02, 0.5) -- (3.01, 0.5); \draw (3, 0.73) node {$\vec{f}$};
\draw [dotted] (1.85, -0.2) -- (1.85, 0.2); \draw [dotted] (4.15, -0.2) -- (4.15, 0.2); \draw [dotted] (2.15, -0.5) -- (2.15, 0.5); \draw [dotted] (3.85, -0.5) -- (3.85, 0.5);

\draw (3.05, -2.2) node {$\left( \Gamma_1, x \right)\#\left( \Gamma_2, y \right)$};
\end{tikzpicture}
\end{center}
\caption{Local picture of \emph{Join} (Case 1)}\label{join}
\end{figure}
Therefore, in this case, the number of boundary components in $\Gamma$ is given by $\lvert \partial\Gamma \rvert = b_1+b_2.$
\end{case}
\begin{case}
In this case, we consider that $\vec{x}\in \partial_1, \cev{x}\in \partial_2$ and $\vec{y}\in \zeta_1, \cev{y}\in \zeta_2$, i.e., the directed edges $\vec{x}, \cev{x}, \vec{y}, \cev{y}$ are in the different boundary components. Therefore, there are directed paths $P, Q, R, S$ as in Case 1, such that 
\begin{align*}
\partial_1 = \vec{x} P=\vec{x}_1 \vec{x}_2P 
,\;\;\; \partial_2 = \cev{x}Q= \cev{x}_2 \cev{x}_1 Q,\;\;\; \zeta_1 = \vec{y} R=\vec{y}_1\vec{y}_2R,\;\;\; \zeta_2 = \cev{y} S=\cev{y}_2\cev{y}_1 S.
\end{align*}
\end{case}
The set of boundary components of $\Gamma$ is given by 
\begin{align*}
\partial\Gamma = \left\lbrace \partial_3, \dots, \partial_{b_1} \right\rbrace \cup \left\lbrace \zeta_3, \dots, \zeta_{b_2} \right\rbrace \cup \left\lbrace \eta_1, \eta_2 \right\rbrace,
\end{align*}
where $\eta_1 = \left( \vec{x}_1\cev{y}_1 \right) S \left( \cev{y}_2 \vec{x}_2 \right) P = \vec{e}\;S\; \vec{f}\;P$ and $\eta_2 = \left( \vec{y}_1\cev{x}_1 \right) Q \left( \cev{x}_2 \vec{y}_2 \right) R = \cev{e}\;Q\; \cev{f}\;R.$ Therefore, we have $\lvert \partial\Gamma \rvert = b_1+b_2-2$ and the proposition holds. 

\begin{case}
In this case, we consider $\vec{x}, \cev{x}$ are in the different boundary component of $\Gamma_1$ and $\vec{y}, \cev{y}$ are in the same boundary component of $\Gamma_2$. Therefore, assume that $\vec{x}\in \partial_1, \cev{x}\in \partial_2$ and $\vec{y}, \cev{y}\in \zeta_1$. As before, we have 
\begin{align*}
\partial_1 = \vec{x} P=\vec{x}_1 \vec{x}_2P 
,\;\;\; \partial_2 = \cev{x}Q= \cev{x}_2 \cev{x}_1 Q,\;\;\; \zeta_1 = R \vec{y} S \cev{y}= R \vec{y}_1\vec{y}_2S\cev{y}_2\cev{y}_1.
\end{align*}
The set of boundary components of $\Gamma$ is given by 
\begin{align*}
\partial\Gamma = \left\lbrace \partial_3, \dots, \partial_{b_1} \right\rbrace \cup \left\lbrace \zeta_2, \dots, \zeta_{b_2} \right\rbrace \cup \left\lbrace \eta \right\rbrace,
\end{align*}
where $\eta= \left( \vec{x}_1 \cev{y}_1\right) S \left( \vec{y}_1 \cev{x}_1\right) Q \left( \cev{x}_2 \vec{y}_2\right) R \left( \cev{y}_2 \vec{x}_2\right) P = \vec{e}S\cev{e} Q\cev{f}R\vec{f}P.$ Therefore, we have $\lvert \partial\Gamma \rvert =b_1+b_2-2.$ 
\end{case}
\begin{case}
In this case, consider that $\vec{x}, \cev{x}\in \partial_1$ and $\vec{y} \in \zeta_1, \cev{y}\in \zeta_2$. This case is the same as Case 3, after interchanging the role of $\Gamma_1$ and $\Gamma_2$. 
\end{case}
Thus the proof of (1) is complete.

\noindent (2) It is easy to see the standard cycle containing $x$ in $\Gamma_1$ and the standard cycle containing $y$ in $\Gamma_2$ join together in $\Gamma$ and the rest of the cycles in $\Gamma_1$ and $\Gamma_2$ remain unchanged in $\Gamma$. Thus the number of standard cycles in $\Gamma$ is $s_1+s_2-1.$

\noindent (3) Now, we calculate the genus of $\Gamma$. The number of vertices in $\Gamma_i$ is $\lvert V_i \rvert$ and edges $2\lvert V_i \rvert$, for each $i=1, 2$. The number of vertices and edges in $\Gamma$ are $\lvert V_1 \rvert + \lvert V_2 \rvert$ and $2\left( \lvert V_1 \rvert + \lvert V_2 \rvert \right).$ By the given condition and using Euler's equation, we have $b_i-\lvert V_i \rvert =2-2g_i$, $i=1,2$ which implies that $b_1+b_2-\left( \lvert V_1 \rvert + \lvert V_2 \rvert \right)= 4-2\left(g_1+g_2\right)$. If $b$ is the number of boundary components of $\Gamma$, then genus $g$ of $\Gamma$ is given by, $$b- \left( \lvert V_1 \rvert + \lvert V_2 \rvert \right) = 2-2g.$$

If $\vec{x}, \cev{x} \in \partial$ and $\vec{y}, \cev{y}\in \eta$ for some $\partial\in \partial\Gamma_1, \eta\in \partial\Gamma_2$, then $b=b_1+b_2.$ In this case, we have $2-2g =4-2\left(g_1+g_2\right)\Rightarrow g=g_1+g_2-1.$ 
In the remaining cases, $b=b_1+b_2-2$. Thus we have, $2-2g = 2-2\left(g_1+g_2\right)\Rightarrow g=g_1+g_2.$
\end{proof}

\begin{corollary}\label{cor:join}
Let $\Gamma_1,\Gamma_2,\text{ and }\Gamma$ be as in Proposition~\ref{prop:join}. Then, the length of each boundary component in $\partial\Gamma \setminus \left( \partial\Gamma_1 \cup \partial \Gamma_2 \right)$ is strictly greater than $2$.
\end{corollary}

%%%%%%%%%%%%%%%%%%%%%%%%  Connected Sum %%%%%%%%%%%%%%%%%%%%%%%%%%%%%%%%%%%%%%%%%%%%

\subsection{Connected sum of fat graphs}\label{cs}
Motivated by the construction used in the proof of Theorem 1.4 in~\cite{BS}, we define a binary operation on fat graphs called \emph{connected sum}. Let $\Gamma_i, i=1,2$, be two $4$-regular fat graphs. Let $w=\{ \vec{e}_1,\vec{e}_2,\vec{e}_3,\vec{e}_4 \}$ and $u =\{ \vec{f}_1,\vec{f}_2,\vec{f}_3,\vec{f}_4 \}$ be two vertices of $\Gamma_1$ and $\Gamma_2$ respectively. Let $\sigma_w=( \vec{e}_1,\vec{e}_2,\vec{e}_3,\vec{e}_4),\sigma_u=( \vec{f}_1,\vec{f}_2,\vec{f}_3,\vec{f}_4)$ be the cyclic orders at $w$ and $u$ respectively. We define the new graph $\Gamma_1\#_{(w,u)} \Gamma_2 = \left( E, \sim, \sigma_1, \sigma_0\right)$, called the connected sum of $\Gamma_1$ and $\Gamma_2$ at the vertices $w$ and $u$, which is described as follows:
\begin{enumerate}
\item The set of directed edges is given by, $$E = \left( E_1 \cup E_2 \setminus \left\lbrace \vec{e}_i, \cev{e}_i, \vec{f}_i, \cev{f}_i|\,1\leq i\leq 4 \right\rbrace \right) \cup \left\lbrace \vec{g}_i, \cev{g}_i|\,1\leq i\leq 4 \right\rbrace,$$ where $E_1$ and $E_2$ are the set of directed edges of $\Gamma_1$ and $\Gamma_2$ respectively and the $\vec{g}_i,\cev{g}_i$ are defined by (see Figure~\ref{sum} for a local picture),  $$\vec{g_i} = \cev{e}_i*\vec{f}_{5-i} \text{ and } \cev{g}_i = \cev{f}_{5-i}*\vec{e}_i\text{ for }1\leq i\leq 4.$$
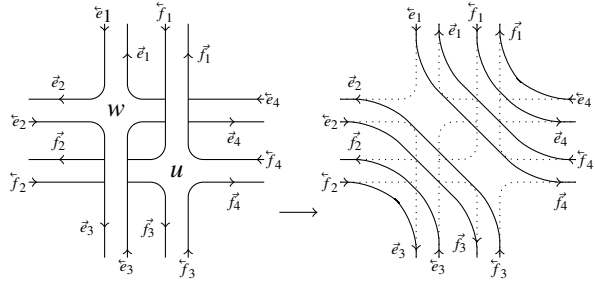
\begin{figure}[htbp]
\begin{center}
\begin{tikzpicture}
\draw [dotted](0,1) -- (1, 1); \draw [dotted] (0,1.3) -- (1, 1.3); \draw [dotted,rounded corners=2mm](1.3,1) -- (1.8, 1) -- (1.8,0); \draw [dotted,rounded corners=2mm](3.1,1) -- (2.1, 1) -- (2.1,0); \draw [dotted,rounded corners=2mm](3.1,1.3) -- (2.1, 1.3) -- (2.1, 3.1); \draw [dotted,rounded corners=2mm](1.3,1.3) -- (1.8, 1.3) -- (1.8,3.1);

\draw [dotted,rounded corners=2mm](1,0) -- (1, 1.8) -- (0, 1.8); \draw [dotted,rounded corners=2mm](1.3,0) -- (1.3, 1.8) -- (1.8, 1.8); \draw [dotted,rounded corners=2mm](0,2.1) -- (1, 2.1) -- (1, 3.1); \draw [dotted,rounded corners=2mm](1.3,3.1) -- (1.3, 2.1) -- (1.8, 2.1); \draw [dotted](2.1,1.8) -- (3.1, 1.8);\draw [dotted](2.1,2.1) -- (3.1, 2.1);
%\draw (1.95, 1.15) node {$u$}; \draw (1.15, 1.95) node {$w$};
\draw [->] (1.3, 3) -- (1.3, 3.01)node [right] {\tiny $\vec{e}_1$}; \draw [->] (1, 3.01) -- (1, 3)node [above] {\tiny $\cev{e}_1$};\draw [->] (2.1, 3) -- (2.1, 3.01)node [right] {\tiny $\vec{f}_1$};\draw [->] (1.8, 3.01) -- (1.8, 3)node [above] {\tiny $\cev{f}_1$};
\draw [->] (0.21, 2.1) -- (0.2, 2.1)node [above] {\tiny $\vec{e}_2$}; \draw [->] (0.1, 1.8) -- (0.11, 1.8)node [left] {\tiny $\cev{e}_2$}; \draw [->] (0.21, 1.3) -- (0.2, 1.3)node [above] {\tiny $\vec{f}_2$}; \draw [->] (0.1, 1) -- (0.11, 1)node [left] {\tiny $\cev{f}_2$}; \draw [->] (1, 0.11) -- (1, 0.1)node [left] {\tiny $\vec{e}_3$}; \draw [->] (1.3, 0.1) -- (1.3, 0.11)node [below] {\tiny $\cev{e}_3$}; \draw [->] (1.8, 0.21) -- (1.8, 0.2)node [left] {\tiny $\vec{f}_3$}; \draw [->] (2.1, 0.1) -- (2.1, 0.11)node [below] {\tiny $\cev{f}_3$};\draw [->] (2.9, 1) -- (2.91, 1)node [below] {\tiny $\vec{f}_4$}; \draw [->] (3.01, 1.3) -- (3, 1.3)node [right] {\tiny $\cev{f}_4$};

\draw [->] (2.9, 1.8) -- (2.91, 1.8)node [below] {\tiny $\vec{e}_4$}; \draw [->] (3.01, 2.1) -- (3, 2.1)node [right] {\tiny $\cev{e}_4$};

\draw [rounded corners=4mm] (0, 1) -- (0.5, 1) -- (1,0.5) -- (1, 0); \draw [rounded corners=4mm] (0, 1.3) -- (0.6, 1.3) -- (1.3,0.6) -- (1.3, 0);

\draw [rounded corners=4mm] (1.8, 3.1) -- (1.8, 2.5) -- (2.5,1.8) -- (3.1, 1.8); \draw [rounded corners=4mm] (2.1, 3.1) -- (2.1, 2.6) -- (2.6,2.1) -- (3.1, 2.1);

\draw [rounded corners=4mm] (1, 3.1) -- (1, 2.5) -- (2.5,1) -- (3.1, 1); \draw [rounded corners=4mm] (1.3, 3.1) -- (1.3, 2.6) -- (2.6,1.3) -- (3.1, 1.3);

\draw [rounded corners=4mm] (0, 1.8) -- (0.5, 1.8) -- (1.8,0.5) -- (1.8, 0); \draw [rounded corners=4mm] (0, 2.1) -- (0.6, 2.1) -- (2.1,0.6) -- (2.1,0);
%%%%%%%%%%%%%%%%%%%%%%%%%%%%%%%%%%%%%%%%%%%%%%%
\draw [->] (-0.8, 0.6) -- (-0.3,0.6);
\draw (-4.1,1) -- (1-4.1, 1); \draw (-4.1,1.3) -- (1-4.1, 1.3); \draw [rounded corners=2mm](1.3-4.1,1) -- (1.8-4.1, 1) -- (1.8-4.1,0); \draw [rounded corners=2mm](3.1-4.1,1) -- (2.1-4.1, 1) -- (2.1-4.1,0); \draw [rounded corners=2mm](3.1-4.1,1.3) -- (2.1-4.1, 1.3) -- (2.1-4.1, 3.1); \draw [rounded corners=2mm](1.3-4.1,1.3) -- (1.8-4.1, 1.3) -- (1.8-4.1,3.1);

\draw [rounded corners=2mm](1-4.1,0) -- (1-4.1, 1.8) -- (-4.1, 1.8); \draw [rounded corners=2mm](1.3-4.1,0) -- (1.3-4.1, 1.8) -- (1.8-4.1, 1.8); \draw [rounded corners=2mm](-4.1,2.1) -- (1-4.1, 2.1) -- (1-4.1, 3.1); \draw [rounded corners=2mm](1.3-4.1,3.1) -- (1.3-4.1, 2.1) -- (1.8-4.1, 2.1); \draw (2.1-4.1,1.8) -- (3.1-4.1, 1.8);\draw (2.1-4.1,2.1) -- (3.1-4.1, 2.1);
\draw (1.95-4.1, 1.15) node {$u$}; \draw (1.15-4.1, 1.95) node {$w$};
\draw [->] (1.3-4.1, 2.7) -- (1.3-4.1, 2.71)node [right] {\tiny $\vec{e}_1$}; \draw [->] (1-4.1, 3.01) -- (1-4.1, 3)node [above] {$\tiny \cev{e}_1$};\draw [->] (2.1-4.1, 2.7) -- (2.1-4.1, 2.71)node [right] {\tiny $\vec{f}_1$};\draw [->] (1.8-4.1, 3.01) -- (1.8-4.1, 3)node [above] {\tiny $\cev{f}_1$};
\draw [->] (0.41-4.1, 2.1) -- (0.4-4.1, 2.1)node [above] {\tiny $\vec{e}_2$}; \draw [->] (0.1-4.1, 1.8) -- (0.11-4.1, 1.8)node [left] {\tiny $\cev{e}_2$}; \draw [->] (0.41-4.1, 1.3) -- (0.4-4.1, 1.3)node [above] {\tiny $\vec{f}_2$}; \draw [->] (0.1-4.1, 1) -- (0.11-4.1, 1)node [left] {\tiny $\cev{f}_2$}; \draw [->] (1-4.1, 0.41) -- (1-4.1, 0.4)node [left] {\tiny $\vec{e}_3$}; \draw [->] (1.3-4.1, 0.1) -- (1.3-4.1, 0.11)node [below] {\tiny $\cev{e}_3$}; \draw [->] (1.8-4.1, 0.41) -- (1.8-4.1, 0.4)node [left] {\tiny $\vec{f}_3$}; \draw [->] (2.1-4.1, 0.1) -- (2.1-4.1, 0.11)node [below] {\tiny $\cev{f}_3$};\draw [->] (2.7-4.1, 1) -- (2.71-4.1, 1)node [below] {\tiny $\vec{f}_4$}; \draw [->] (3.01-4.1, 1.3) -- (3-4.1, 1.3)node [right] {\tiny $\cev{f}_4$};

\draw [->] (2.7-4.1, 1.8) -- (2.71-4.1, 1.8)node [below] {\tiny $\vec{e}_4$}; \draw [->] (3.01-4.1, 2.1) -- (3-4.1, 2.1)node [right] {\tiny $\cev{e}_4$};
\end{tikzpicture}
\end{center}
\caption{Connected sum (local picture)}\label{sum}
\end{figure}
\item The fixed point free involution $\sigma_1$ is defined as usual $\sigma\left( \vec{a} \right) = \cev{a}$ for all $\vec{a}\in E$.

\item Let $v = \left\lbrace \vec{h}_i: i=1, \dots, 4 \right\rbrace\in V_1\cup V_2 \setminus \{w,u\}$, where $V_1$ and $V_2$ are the set of vertices of $\Gamma_1$ and $\Gamma_2$ respectively. We define $v'=\left\lbrace {\vec{h}_i}^{\,'} : i=1,\dots, 4\right\rbrace$, where
\[
{\vec{h}_i}^{\,'}=
\begin{cases}
\vec{h}_i& \text{ if } \vec{h}_i\in E_1 \cup E_2 \setminus \left\lbrace  \vec{e}_i, \cev{e}_i, \vec{f}_i, \cev{f}_i|\,1\leq i\leq 4 \right\rbrace \text{ and }\\
\vec{g}_i \text{ and }\cev{g}_{5-i} & \text{if } h_i=\cev{e}_i \text{ and } \cev{f}_i \text{ respectively.}
\end{cases}
\]
Then the set equivalence classes of the relation $\sim$ is $V= \left\lbrace v': v\in V_1\cup V_2\setminus \{w,u\}\right\rbrace.$
\item If $\sigma_v = \left( \vec{h}_1, \vec{h}_2, \vec{h}_3, \vec{h}_4\right)$ is the cyclic order at the vertex $v\in V_1\cup V_2\setminus\{w,u\}$, then we define $\sigma_{v'} = \left( {\vec{h}_1}^{\,'}, {\vec{h}_2}^{\,'}, {\vec{h}_3}^{\,'}, {\vec{h}_4}^{\,'}\right)$ is the cyclic order at the vertex $v'\in V$. Thus $\sigma_0=\prod\limits_{v'\in V}\sigma_{v'}.$
\end{enumerate}

For a fat graph $\Gamma=\left( E, \sim, \sigma_1, \sigma_0\right)$ and $\eta\in \partial\Gamma$, we define $\chi_{\eta}:E\to \mathbb{R}$  as follows:
$$
\chi_{\eta}(\vec{e})=
\begin{cases}
1& \text{if }\vec{e}\in \eta,\text{ and}\\
0 & \text{otherwise.}
\end{cases}
$$
\begin{prop}\label{connected_sum}
Let $\Gamma_i, i=1, 2,$ be two $4$-regular fat graphs of genus $g_i$ with $b_i$ boundary components and $s_i$ standard cycles. Let $w=\{ \vec{e}_1,\vec{e}_2,\vec{e}_3,\vec{e}_4 \}$ and $u =\{ \vec{f}_1,\vec{f}_2,\vec{f}_3,\vec{f}_4 \}$ be two vertices of $\Gamma_1$ and $\Gamma_2$ respectively.  Suppose that for $1\leq i\leq 4,\,\vec{e}_i,\cev{e}_i\in \partial$ for some $\partial\in \partial \Gamma_1$. Then the connected sum $\Gamma =  \Gamma_1\#_{(w,u)} \Gamma_2 $ of $\Gamma_1$ and $\Gamma_2$ at $w, u$ has the following properties:
\begin{enumerate}
\item The number of boundary components in $\Gamma$ is given by 
\[
\lvert \partial\Gamma \rvert=
\begin{cases}
b_1+b_2+2 & \text{ if } \sum\limits_{i=1}^4\chi_{\eta}(\vec{f}_i)\chi_{\eta}(\cev{f}_i)=4\text{ for some }\eta\in \partial\Gamma_2,\\
b_1+b_2-4 & \text{ if } \sum\limits_{\substack{\eta\in \partial\Gamma_2\\1\leq i\leq 4}}\chi_{\eta}(\vec{f}_i)\chi_{\eta}(\cev{f}_{i})=0,\text{ and } \sum\limits_{\substack{\eta\in \partial\Gamma_2\\1\leq i\leq 4}}\chi_{\eta}(\cev{f}_i)\chi_{\eta}(\vec{f}_{i+1})=4,\\
b_1+b_2 & \text{ if }\sum\limits_{\substack{\eta\in \partial\Gamma_2\\1\leq i\leq 4}}\chi_{\eta}(\vec{f}_i)\chi_{\eta}(\cev{f}_i)=2\text{ and}\\
b_1+b_2-2 & \text{ otherwise. } 
\end{cases}
\]

\item The number of standard cycles in $\Gamma$ is $s_1+s_2-2$.

\item The genus of $\Gamma$ is given by
\[
g(\Gamma)=
\begin{cases}
g_1+g_2-3 & \text{ if } \sum\limits_{i=1}^4\chi_{\eta}(\vec{f}_i)\chi_{\eta}(\cev{f}_i)=4\text{ for some }\eta\in \partial\Gamma_2,\\
g_1+g_2 & \text{ if } \sum\limits_{\substack{\eta\in \partial\Gamma_2\\1\leq i\leq 4}}\chi_{\eta}(\vec{f}_i)\chi_{\eta}(\cev{f}_{i})=0,\text{ and } \sum\limits_{\substack{\eta\in \partial\Gamma_2\\1\leq i\leq 4}}\chi_{\eta}(\cev{f}_i)\chi_{\eta}(\vec{f}_{i+1})=4,\\
g_1+g_2-2 & \text{ if }\sum\limits_{\substack{\eta\in \partial\Gamma_2\\1\leq i\leq 4}}\chi_{\eta}(\vec{f}_i)\chi_{\eta}(\cev{f}_i)=2\text{ and}\\
g_1+g_2-1 & \text{ otherwise. } 
\end{cases}
\]
\end{enumerate}
\end{prop}

\begin{proof}
Let us assume that $\partial\Gamma_1=\left\lbrace \partial_1, \dots, \partial_{b_1} \right\rbrace$ and $\partial\Gamma_2=\left\lbrace \zeta_1, \dots, \zeta_{b_2} \right\rbrace$. Without loss of generality, we take $\partial_1 =\cev{e}_1\vec{e}_2P\cev{e}_2\vec{e}_3Q\cev{e}_3\vec{e}_4R\cev{e}_4\vec{e}_1S$ for some directed paths $P,Q,R,S$ on the boundary of $\Gamma_1$. We begin by considering the case $$ \sum\limits_{\substack{\eta\in \partial\Gamma_2\\1\leq i\leq 4}}\chi_{\eta}(\vec{f}_i)\chi_{\eta}(\cev{f}_{i})=0,\text{ and } \sum\limits_{\substack{\eta\in \partial\Gamma_2\\1\leq i\leq 4}}\chi_{\eta}(\cev{f}_i)\chi_{\eta}(\vec{f}_{i+1})=4,$$ the arguments in the remaining cases are similar.

In this case, we have assumed that the directed paths $\cev{f}_1\vec{f}_2,\cev{f}_2\vec{f}_3,\cev{f}_3\vec{f}_4$ and $\cev{f}_4\vec{f}_1$ are in different boundary components of $\Gamma_2$. Therefore, we assume that $\cev{f}_1\vec{f}_2\in \zeta_1,\cev{f}_2\vec{f}_3\in \zeta_2,\cev{f}_3\vec{f}_4\in \zeta_3,\text{ and }\cev{f}_4\vec{f}_1\in \zeta_4$. Thus there are directed paths $P', Q',R',\text{ and } S'$ on the boundary of $\Gamma_2$  respectively, such that 
\begin{align*}
\zeta_1 = \cev{f}_1\vec{f}_2P',\zeta_2 = \cev{f}_2\vec{f}_3Q',\zeta_3 = \cev{f}_3\vec{f}_1 R'\text{ and } \zeta_4 = \cev{f}_4\vec{f}_1S'.
\end{align*}
The set of boundary components of the  fat graph $\Gamma =  \Gamma_1\#_{(w,u)} \Gamma_2 $ is given by 
\begin{align*}
\partial\Gamma = \left\lbrace \partial_2, \dots, \partial_{b_1} \right\rbrace \cup \left\lbrace \zeta_5, \dots, \zeta_{b_2} \right\rbrace \cup \left\lbrace \eta \right\rbrace,
\end{align*}
where
\begin{eqnarray*}
\eta&=&(\cev{e}_1*\vec{f}_4)R'(\cev{f}_3*\vec{e}_2)P(\cev{e}_2*\vec{f}_3)Q'(\cev{f}_2*\vec{e}_3)Q(\cev{e}_3*\vec{f}_2)P'(\cev{f}_1*\vec{e}_4)R(\cev{e}_4*\vec{f}_1)
S'(\cev{f}_4*\vec{e}_1)S\\&=&\cev{g}_1R'\cev{g}_2P\vec{g}_2Q'\cev{g}_3Q\vec{g}_3P'\cev{g}_4R\vec{g}_4
S'\cev{g}_1S .
\end{eqnarray*}

Therefore, we have $\lvert \partial\Gamma \rvert=b_1+b_2-4.$ 

To prove (2),  assume that $S(\Gamma_1)=\left\lbrace T_1, \dots, T_{s_1} \right\rbrace, S(\Gamma_2)=\left\lbrace U_1, \dots, U_{s_2} \right\rbrace$ be the set of standard cycles of $\Gamma_1$ and $\Gamma_2$ respectively. Without loss of generality assume that there exist paths $T_i'\in T_i,U_i'\in U_i$, for $i=1,2$ such that $$T_1 = e_1T_1'e_3, T_2=e_2T_2'e_4,U_1=f_1U_1'f_3,\text{ and }U_2=f_2T_2'f_4.$$
The set of standard cycles of $\Gamma$ are given by
$$ S(\Gamma) = \left\lbrace T_3, \dots, T_{s_1} \right\rbrace \cup \left\lbrace U_3, \dots, U_{s_2} \right\rbrace \cup \left\lbrace V_1,V_2 \right\rbrace,
$$ where $ V_1=g_1U_2'g_3T_1',\text{ and }V_2=g_2U_1'g_4T_2',$ which conclude the proof of (2). Finally, a simple Euler characteristic argument coupled with (1) yields (3).
\end{proof}

%%%%%%%%%%%%%%%%%%%%%%  Plumbing  %%%%%%%%%%%%%%%%%%%%%%%%%%%%%%%%%%%%%%%%%%%%%%%%%%%

\subsection{Plumbing of fat graphs}\label{pl}
Let $\Gamma_i, \; i=1,2,$ be two fat graphs and $x=\{\vec{x}, \cev{x}\}$, $y=\{\vec{y}, \cev{y}\}$ be two undirected edges of $\Gamma_1, \Gamma_2$ respectively. The \emph{plumbing} of $\Gamma_1$ and $\Gamma_2$ along the edges $x, y$ is the fat graph $(\Gamma_1\#\Gamma_2)_{(x, y)}=(E', \sim, \sigma_1', \sigma_0')$ defined by the following (for a local picture of a plumbing, we refer to Figure~\ref{plumb}):
\begin{enumerate}
\item We split $x=\{\vec{x}, \cev{x}\}$ into two edges $x_i=\{\vec{x}_i, \cev{x}_i\}, i=1,2$ and $y$ into $y_i=\{\vec{y}_i, \cev{y}_i\}, i=1,2$. Define, $$ E' = E(\Gamma_1)\cup E(\Gamma_2) \cup \{\vec{x}_i, \cev{x}_i, \vec{y}_i, \cev{y}_i| i=1,2\} \setminus \{\vec{x}, \cev{x}, \vec{y}, \cev{y}\}.$$

\item  The involution $\sigma_1'$ is as usual, $\sigma_1'(\vec{e})=\cev{e}$ for all $\vec{e}\in E'$.

\item Now, we define the equivalence classes of  $\sim$, equivalently the set $V'$ of vertices: let $v=\{\vec{e}_1, \vec{e}_2, \dots, \vec{e}_k\}\in V(\Gamma_1)\bigcup V(\Gamma_2)$. Define, $v'=\{{\vec{e}_1}^{\,'}, {\vec{e}_2}^{\,'}, \dots, {\vec{e}_k}^{\,'}\}$, where 
$$
{\vec{e}_i}^{\,'}=
\begin{cases}
\vec{e}_i, & \text{ if } \vec{e}_i\neq \vec{x}, \vec{y}, \cev{x}, \cev{y} \text{ and}\\
\vec{x}_1, \cev{x}_2, \vec{y}_1\text{ or } \cev{y}_2 & \text{ if } \vec{e}_i = \vec{x}, \cev{x}, \vec{y} \text{ or } \cev{y}\text{ respectively. }
\end{cases}
$$
Define $V'=\{v'|\; v\in V(\Gamma_1)\bigcup V(\Gamma_2)\}\cup \{u\},$ where $u=\{\cev{x}_1, \cev{y}_1, \vec{x}_2, \vec{y}_2\}$. 

\item Let $\sigma_v=(\vec{e}_1, \dots, \vec{e}_k)$ be the cyclic order at a vertex $v\in V(\Gamma_1)\bigcup V(\Gamma_2)$. Then we define ${\sigma'}_{v'}=({\vec{e}_1}^{\,'}, \dots, {\vec{e}_k}^{\,'})$ and further, $\sigma_u'=(\cev{x}_1, \cev{y}_1, \vec{x}_2, \vec{y}_2)$. The fat graph structure is given by $\sigma_0'=\prod\limits_{v'\in V'} {\sigma'}_{v'}.$
\end{enumerate}

\begin{dfn}
The graph $(\Gamma_1\#\Gamma_2)_{(x, y)}$ is said to be obtained by plumbing $\Gamma_1$ and $\Gamma_2$ along the edges $x, y$.
\end{dfn}

\begin{prop}\label{plumb:boundary}
Let $\Gamma_i, i=1, 2$ be fat graphs and $x, y$ be two edges of $\Gamma_1, \Gamma_2$ respectively. If $b_i=|\partial\Gamma_i|$, $i=1,2$, then the number of boundary components in $(\Gamma_1\#\Gamma_2)_{(x, y)}$ is given by $$|\partial (\Gamma_1\#\Gamma_2)_{(x, y)}|=
\begin{cases}
b_1+b_2-3, & \text{ if $\vec{x}, \vec{y}, \cev{x}, \cev{y}$ are in the different boundary components and }\\
b_1+b_2-1, & \text{ otherwise. }
\end{cases}$$ 
\end{prop}

\begin{proof}
Let $\left\lbrace\partial_i| i=1, \dots, b_1\right\rbrace$ and $\left\lbrace \zeta_j| j=1, \dots, b_2\right\rbrace$ be the sets of boundary components of $\Gamma_1$ and $\Gamma_2$ respectively. There are four cases to  be considered. 

\subsection*{Case 1} Consider, $\vec{x}, \cev{x}$ are in the same boundary component of $\Gamma_1$, say $\partial_1$ and $\vec{y}, \cev{y}$ are in $\zeta_1$. Then one can write $\partial_1 = \vec{x}P_1\cev{x}P_2$ and $\zeta_1=\vec{y}Q_1\cev{y}Q_2$, where $P_i, Q_i, i=1,2$ are paths in the boundary (Figure~\ref{plumb}). Then we have, $\partial(\Gamma_1\#\Gamma_2)_{(x, y)} = \left\lbrace  \partial_i| i=2, \dots, b_1 \right\rbrace \cup  \left\lbrace \zeta_j | j=2, \dots, b_2 \right\rbrace \cup \left\lbrace \eta \right\rbrace,$ where $\eta = \vec{x}_1 \cev{y}_1 Q_2 \vec{y}_1 \vec{x}_2 P_1  \cev{x}_2 \vec{y}_2 Q_1 \cev{y}_2 \cev{x}_1 P_2$. Hence, we have $$ |\partial (\Gamma_1 \# \Gamma_2)_{(x, y)} | = b_1+b_2-1.$$
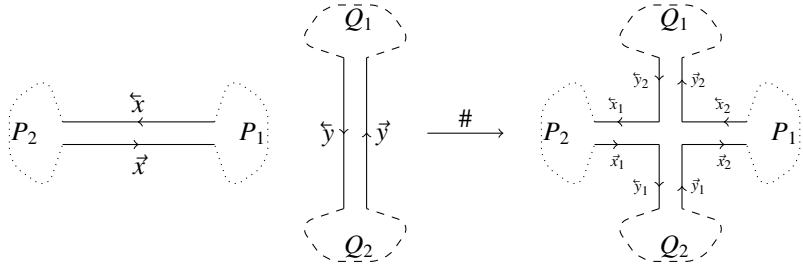
\begin{figure}[htbp]
\begin{center}
\begin{tikzpicture}
\draw (-5, -0.15) -- (-3, -0.15); \draw (-5, 0.15) -- (-3, 0.15); \draw [dotted, rounded corners = 2mm] (-3, 0.15) -- (-2.8, 0.75) -- (-2.3, 0.4)-- (-2.3, -0.4) -- (-2.8, -0.75) -- (-3, -0.15); \draw [dotted, rounded corners = 2mm] (-5, 0.15) -- (-5.2, 0.75) -- (-5.7, 0.4) -- (-5.7, -0.4) -- (-5.2, -0.75) -- (-5, -0.15);
\draw [->] (-4, 0.15) -- (-4.01, 0.15)node [above] {$\cev{x}$}; \draw [->] (-4.01, -0.15) -- (-4, -0.15)node [below] {$\vec{x}$}; \draw (-5.5, 0) node {$P_2$}; \draw (-2.5, 0) node {$P_1$};
\draw (1.5, 0) node {$P_2$}; \draw (4.5, 0) node {$P_1$};

\draw (-1.3, -1) -- (-1.3, 1); \draw (-1, -1) -- (-1, 1); \draw [dashed, rounded corners=2mm] (-1.3, -1) -- (-1.9, -1.2) -- (-1.55, -1.7) -- (-0.75,-1.7) -- (-0.4,-1.2) -- (-1, -1); \draw [dashed, rounded corners = 2mm] (-1.3, 1) -- (-1.9, 1.2) -- (-1.55, 1.7) -- (-0.75, 1.7) -- (-0.4, 1.2) -- (-1, 1);
\draw [->] (-1.3, 0) -- (-1.3, -0.01)node [left] {$\cev{y}$}; \draw [->] (-1, 0) -- (-1, 0.01)node [right] {$\vec{y}$}; \draw (-1.1, 1.5) node {$Q_1$}; \draw (3.05, 1.5) node {$Q_1$}; \draw (-1.1, -1.5) node {$Q_2$}; \draw (3.05, -1.5) node {$Q_2$};

\draw [->] (-0.2, 0) -- (0.8, 0); \draw (0.3, 0.2) node {$\#$};

\draw (2, -0.15) -- (2.85, -0.15); \draw [->] (2.3, -0.15) -- (2.31, -0.15) node [below] {\tiny $\vec{x}_1$}; \draw [->] (2.31, 0.15) -- (2.3, 0.15) node [above] {\tiny $\cev{x}_1$}; \draw (3.15, -0.15) -- (4, -0.15); \draw [->] (3.71, 0.15) -- (3.7, 0.15) node [above] {\tiny $\cev{x}_2$}; \draw [->] (3.7, -0.15) -- (3.71, -0.15) node [below] {\tiny $\vec{x}_2$}; \draw (2, 0.15) -- (2.85, 0.15); \draw (3.15, 0.15) -- (4, 0.15); \draw [dotted, rounded corners = 2mm] (4, 0.15) -- (4.2, 0.75) -- (4.7, 0.4)-- (4.7, -0.4) -- (4.2, -0.75) -- (4, -0.15); \draw [dotted, rounded corners = 2mm] (2, 0.15) -- (1.8, 0.75) -- (1.3, 0.4) -- (1.3, -0.4) -- (1.8, -0.75) -- (2, -0.15);

\draw (2.85, -1) -- (2.85, -0.15); \draw (2.85, 0.15)-- (2.85, 1); \draw [->] (2.85, -0.7) -- (2.85, -0.71) node [left] {\tiny $\cev{y}_1$}; \draw [->] (3.15, -0.71) -- (3.15, -0.7) node [right] {\tiny $\vec{y}_1$}; \draw [->] (2.85, 0.71) -- (2.85, 0.7) node [left] {\tiny $\cev{y}_2$}; \draw [->] (3.15, 0.7) -- (3.15, 0.71) node [right] {\tiny $\vec{y}_2$};  \draw (3.15, -1) -- (3.15, -0.15); \draw (3.15, 0.15)-- (3.15, 1); \draw [dashed, rounded corners=2mm] (2.85, -1) -- (2.25, -1.2) -- (2.6, -1.7) -- (3.4,-1.7) -- (3.75,-1.2) -- (3.15, -1); \draw [dashed, rounded corners = 2mm] (-1.3+4.15, 1) -- (2.25, 1.2) -- (2.6, 1.7) -- (3.4, 1.7) -- (3.75, 1.2) -- (3.15, 1);
\end{tikzpicture}
\end{center}
\caption{Local picture of plumbing (Case 1)}\label{plumb}
\end{figure}

\subsection*{Case 2} In this case, we consider $\vec{x}\in \partial_1, \cev{x_2}\in \partial_2$ and $\vec{y}\in \zeta_1, \cev{y}\in \zeta_2$. We write, $\partial_1=\vec{x} P_1, \partial_2 = \cev{x} P_2$ and $\zeta_1=\vec{y} Q_1, \zeta_2 = \cev{y} Q_2$. We have $$\partial(\Gamma_1\#\Gamma_2)_{(x, y)} = \left\lbrace \partial_i | i = 3, \dots, b_1\right\rbrace \cup \left\lbrace \zeta_j | j=3, \dots, b_2 \right\rbrace \cup \left\lbrace \eta \right\rbrace,$$ where $\eta = \vec{x}_1 \cev{y}_1 Q_2 \cev{y}_2 \cev{x}_1P_2 \cev{x}_2 \vec{y}_2 Q_1 \vec{y}_1\vec{x}_2 P_1$, which implies that  $$ |\partial (\Gamma_1 \# \Gamma_2)_{(x, y)} | = b_1+b_2-3.$$

\subsection*{Case 3} Consider $\vec{x}, \cev{x}\in \partial_1$ and $\vec{y}\in \zeta_1, \cev{y}\in \zeta_2$. Similarly, as in Case 1 and Case 2, we can write $\partial_1=\vec{x} P_1 \cev{x} P_2, \zeta_1=\vec{y}_1 Q_1$, $\zeta_2 = \cev{y} Q_2$. Then the boundary of the new fat graph is given by $$\partial(\Gamma_1\#\Gamma_2)_{(x, y)}=\left\lbrace \partial_i |i=2, \dots, b_1 \right\rbrace \cup \left\lbrace \zeta_j | j=3, \dots, b_2 \right\rbrace \cup \left\lbrace \eta_1, \eta_2 \right\rbrace,$$ where $\eta_1=\vec{x}_1 \cev{y}_1 Q_2 \cev{y}_2 \cev{x}_1 P_2, \eta_2=\vec{x}_2 P_1 \cev{x}_2 \vec{y}_2 Q_1 \vec{y}_1$. This implies that  $$|\partial(\Gamma_1\#\Gamma_2)_{(x, y)}|=b_1+b_2-1.$$ 

\subsection*{Case 4} Here, assume $\vec{x}\in \partial_1, \cev{x}\in \partial_2$ and $\vec{y}, \cev{y}\in \zeta_1$. This case is the same as Case 3, after interchange the role of $\Gamma_1$ and $\Gamma_2$, hence we are done.
\end{proof}

%%%%%%%%%% SECTION 3: Proof of theorem 1.1 %%%%%%%%%

\section{Growth of U$_{g,b}$}
In this section, we prove Theorem~\ref{res:1}. To this end, we begin by proving an upper bound on $\U_{g,b}$.
\begin{lemma}\label{lem:upper_bound}
$\U_{g,b}\leq 2g+b-1,$ for all $g,b\in \mathbb{N}$.
\end{lemma}
\begin{proof}
Let $\Omega = \left\lbrace \gamma_1, \gamma_2, \dots, \gamma_k \right\rbrace$ be a  filling of $F_g$. We regard the union $G=\bigcup_{i=1}^k \gamma_i$ as a connected $4$-regular decorated fat graph on $F_g$, where the intersection points $\gamma_i\cap \gamma_j, i\neq j \in \left\lbrace 1,2,\dots, k \right\rbrace $ are the vertices and the sub-arcs of $\gamma_i$'s connecting the vertices are the edges. The fat graph structure is uniquely determined by the orientation of the surface. By Euler's equation and valency condition, we note that the number of vertices and edges in $G$ are $2g-2+b$ and $4g-4+2b$ respectively. 
%If we cut $F_g$ along the graph $G$, we obtain a $\left( 8g-4 \right)$-sided polygon. 

We rename the elements of the filling as  $\Omega=\left\lbrace \delta_1, \delta_2, \dots, \delta_k\right\rbrace$ such that $\left(\bigcup_{i=1}^m\delta_i\right)\bigcap \delta_{m+1} \neq \emptyset,$ for all $ m=1,2\dots, k-1.$ We begin with considering $\delta_1=\gamma_1$. The graph $G$ is connected implies that $\delta_1$ has non-empty intersection with $\bigcup_{i=2}^k\gamma_i$. Therefore, there is a curve $\gamma_{i_2}\in \Omega\setminus \left\lbrace \delta_1 \right\rbrace $ such that $\gamma_{i_2}\cap \delta_1\neq \emptyset$. We choose $\delta_2=\gamma_{i_2}$. In this fashion, we rename the elements of $\Omega$.   

For each $m=2,\dots, k,$ the union $ G_m = \bigcup_{i=1}^{m} \delta_i$ is a connected $4$-regular decorate fat graph on $F_g$ and $G=G_k$. Let $V_m$ denote the number of vertices in $G_m$. Then $V_{m+1}\geq V_m+1\text{ and }V_2\geq 1$, which gives $V_{m+1}\geq m.$ In particular, $V_k\geq k-1$ and hence, $2g+b-1\geq k$. Therefore, we have $U_g\leq 2g+b-1$.
\end{proof}
We prove Theorem~\ref{res:1} by induction on $b$. To apply the induction step, we need to prove Theorem~\ref{res:1} for minimal fillings i.e, the case when $b=1$.
\begin{prop}\label{prop:minimal_filling}
$\U_{g,1}=2g$, for all $g\in \mathbb{N}$.
\end{prop}
\begin{proof}
By Lemma~\ref{lem:upper_bound}, we have $U_{g,1}\leq 2g$. To complete the proof, we construct an example of minimally intersecting filling set $\Omega_{g}^{\max}$ with $2g$ curves.   

Consider the $4$-regular decorated fat graph $\Gamma_g=\left( E, \sim, \sigma_1, \sigma_0 \right)$ described below.
\begin{enumerate}
\item $E=\left\lbrace \vec{e}_i, \cev{e}_i\;|\;i=1,2,\dots, 2\left( 2g-1 \right) \right\rbrace$.
\item The set of equivalence classes of $\sim$ is $V=\left\lbrace v_j|\; j=1, \dots, 2g-1 \right\rbrace$, where 
\[
  v_j = \left\{\def\arraystretch{1.2}%
  \begin{array}{@{}c@{\quad}l@{}}
    \left\lbrace \vec{e}_1, \vec{e}_2, \cev{e}_1, \cev{e}_3 \right\rbrace & \text{if $j=1$},\\
    \left\lbrace \vec{e}_{2j-1}, \vec{e}_{2j}, \cev{e}_{2j-2}, \cev{e}_{2j+1}\right\rbrace & \text{if $2\leq j \leq 2g-2$} \text{ and}\\
    \left\lbrace \vec{e}_{4g-3}, \vec{e}_{4g-2}, \cev{e}_{4g-4}, \cev{e}_{4g-2}\right\rbrace & \text{if $j=2g-1$}.\\
    
  \end{array}\right.
\]
\item The fixed point free involution $\sigma_1$ is defined by $\sigma_1 \left( \vec{e}_i \right) = \cev{e}_i$ for $i=1,2,\dots, 2 \left( 2g-1 \right).$
\item The permutation $\sigma_0$ is given by $\sigma_0=\prod\limits_{j=1}^{2g-1}\sigma_{v_j}$, where 
\[
 \sigma_{v_j} = \left\{\def\arraystretch{1.2}%
  \begin{array}{@{}c@{\quad}l@{}}
    \left( \vec{e}_1, \vec{e}_2, \cev{e}_1, \cev{e}_3 \right) & \text{if $j=1$},\\
    \left( \vec{e}_{2j-1}, \vec{e}_{2j}, \cev{e}_{2j-2}, \cev{e}_{2j+1} \right) & \text{if $2\leq j \leq 2g-2$}\text{ and}\\
    \left( \vec{e}_{4g-3}, \vec{e}_{4g-2}, \cev{e}_{4g-4}, \cev{e}_{4g-2} \right) & \text{if $j=2g-1$}.\\
    
  \end{array}\right.
\]
\end{enumerate}

Now, we count the number of boundary components of the graph. For each $g\geq 2$, we define,
\begin{align*}
P\left( g \right) &= \left( \cev{e}_3 *\vec{e}_4 \right) * \left( \cev{e}_7 * \vec{e}_8 \right) * \cdots * \left( \cev{e}_{4g-5} * \vec{e}_{4g-4} \right) * \cev{e}_{4g-2},\\ 
Q \left( g \right) &= \left( \cev{e}_{4g-4} * \cev{e}_{4g-6} * \dots * \cev{e}_4 * \cev{e}_2 \right) * \cev{e}_1,\\ 
R \left( g \right) &= \vec{e}_2 * \left( \cev{e}_5 * \vec{e}_6 \right) * \left( \cev{e}_9 * \vec{e}_{10} \right) * \dots * \left( \cev{e}_{4g-3} * \vec{e}_{4g-2} \right)\text{ and} \\ 
S \left( g \right) &= \vec{e}_{4g-3} * \vec{e}_{4g-5} * \cdots * \vec{e}_5 * \vec{e}_3 * \vec{e}_1. 
\end{align*}
The boundary of the fat graph $\Gamma_g$ is given by $\partial \Gamma_g= P \left( g \right) * Q \left( g \right) * R \left( g \right) * S \left( g \right).$ By attaching a topological disc along the boundary of $\Gamma_g$, we obtain closed oriented surface $F_g$. Then the set of standard cycles of $\Gamma_g$, we denote by $\Omega_g^{\max}$, is a minimally intersecting filling set of $F_g$ of size $2g$. This completes the proof.
\end{proof}

We are now ready to prove Theorem~\ref{res:1}.
\begin{proof}[Proof of Theorem~\ref{res:1}]
Lemma~\ref{lem:upper_bound} implies that $U_{g,b}\leq 2g-1+b$. To complete the proof, we show the existence of a filling system of size $2g-1+b$ with $b$ complementary discs.  We prove this result by induction on $b$. For $b=1$, the result follows from Propostion~\ref{prop:minimal_filling}, and so we assume $b\geq 2$. Let $\Gamma_{g}$ be the $4$-regular fat graph, constructed in the proof of Proposition~\ref{prop:minimal_filling}, which corresponds to a minimal filling of size $2g$.  Consider $ \Gamma^1 = \left( \Gamma_g, x \right) \# \left( \Gamma_1, y\right)$, where $x=\{\vec{x},\cev{x}\}$ and $y=\{\vec{y},\cev{y}\}$ are edges of $\Gamma_g$ and $\Gamma_1$ respectively. It follows from (1)-(3) Proposition~\ref{prop:join} that $\Gamma^1$ corresponds to a filling system of size $2g+1$ such that the complement is a pair of disjoint discs. If $b=2$, we are done, else we define $\Gamma^2= \left( \Gamma^1, z \right) \# \left( \Gamma_1, y\right)$, where $z=\{\vec{z},\cev{z}\}$ is an edge of $\Gamma^1$ such that $\vec{z},\cev{z}\in  \eta$ for some $\partial \Gamma^1$, and its existence is assured by the proof of  Proposition~\ref{prop:join}. By the same argument, $\Gamma^2$ gives a filling system of size $2g+2$ such that the complement is a triple of disjoint discs. Continuing this construction $b-1$ times, we will get a fat graph $\Gamma^{b-1}$ which realises a filling system of size $2g+b-1$ such that the complement is a disjoint union of $b$ discs. This completes the proof.
\end{proof}

%%%%% SECTION 4: Minimal fillings of all possible sizes %%%%%%%%%%%

\section{Minimal filling systems of all possible sizes} 

In this section, we prove a special case of Theorem~\ref{res:2}, where $b=1$, i.e., the filling systems are minimal. In section~\ref{sec:5}, this will be used as an initial step for induction argument in proving Theorem~\ref{res:2} for generic fillings. More precisely, we prove the proposition below:
\begin{prop}\label{prop:minimal}
Let $g\geq 2$ be an integer. For each $s$ satisfying $L_{g,1} \leq s \leq U_{g, 1}$, there exists a minimal filling system $\Omega_s$ of $F_g$ with $|\Omega_s|=s.$ 
\end{prop}
Before proceeding further, we give an example of fat graph of genus three, which will be instrumental in the construction of a minimal filling triple on any genus surface.

\begin{exa}\label{eg:triple_3}
Consider the  fat graph $\Gamma_0$, given in Figure~\ref{Fig:exa2}.
%\begin{enumerate}
%\item $E=\{\vec{f}_i, \cev{f}_i\;|\;i=1,2,\dots, 10\}$.
%\item The set of equivalence classes of $\sim$ is $V=\{w_1, w_2, \dots, w_5\}$, where 
%$ w_1=(\vec{f}_1, \vec{f}_2, \vec{f}_3, \vec{f}_4),\,\\w_2=(\cev{f}_3, \vec{f}_5, \vec{f}_6, \vec{f}_7),\,w_3=(\cev{f}_4, \cev{f}_6, \vec{f}_8, \vec{f}_9),\,w_4=(\cev{f}_5, \vec{f}_{10}, \cev{f}_7, \cev{f}_9),\,w_5=(\cev{f}_1, \cev{f}_2, \cev{f}_{10}, \cev{f}_8).$
%\item The fixed point free involution $\sigma_1$ is defined by $\sigma_1(\vec{f}_i)=\cev{f}_i$, for $i=1,2,\dots, 10.$
%\item The permutation $\sigma_0$ is given by $\small \displaystyle \sigma_0=\prod_{i=1}^{5}w_i$.
%\end{enumerate}
\begin{figure}[htbp]
\begin{center}
\begin{tikzpicture}
\draw [rounded corners=2mm](5, 1) -- (5, 0) -- (7.3,0) -- (7.3, 3.1) -- (6.3, 3.1) -- (6.3, 4.9) -- (5.3, 4.9) -- (5.3, 6.6)-- (0.7, 6.6) -- (0.7, 5.9);

\draw [rounded corners=2mm](5.3, 1) -- (5.3, 0.3) -- (7, 0.3) -- (7, 2.8) -- (6.3 ,2.8) -- (6.3, 1) -- ( 4.3, 1) -- (4.3, 0) -- (0,0) -- (0, 5.9) -- (3.3,5.9) -- (3.3, 4.9)--(5, 4.9) -- (5, 6.3) -- (1, 6.3) -- (1, 5.9);

\draw [rounded corners=2mm](1, 5.6) -- (1, 3.9) -- (1.7,3.9)-- (1.7, 4.9) -- (3, 4.9) -- (3, 5.6) -- (0.3,5.6) -- (0.3, 0.3) -- (4,0.3) -- (4, 1) -- (1.7, 1) -- (1.7, 3.6)-- (0.7, 3.6) --(0.7, 5.6);

\draw [rounded corners=2mm]  (5, 1.3) -- (5, 2) -- (4.3,2) -- (4.3, 1.3) -- (6,1.3) -- (6, 2.8) -- (3,2.8) -- (3, 4.6) -- (2,4.6) -- (2,3.9) -- (3, 3.9);

\draw [rounded corners=2mm]  (5.3, 1.3) -- (5.3, 2.3) -- (4,2.3) -- (4,1.3)--(2, 1.3) -- (2, 3.6) -- (3, 3.6);

\draw [rounded corners=2mm]  (3.3,3.6) -- (5.3, 3.6) -- (5.3, 4.6) -- (6, 4.6) -- (6, 3.1) -- (3.3, 3.1) -- (3.3,4.6) -- (5, 4.6) -- (5, 3.9) -- (3.3,3.9);

\draw [->] (6, 2.5) -- (6, 2.49)node [left] {$\vec{f}_1$}; \draw [->] (6.8, 2.8) -- (6.81, 2.8)node [below] {$\vec{f}_2$}; \draw [->] (6.3, 3.6) -- (6.3, 3.61)node [right] {$\vec{f}_3$};\draw [->] (5.51, 3.1) -- (5.5, 3.1)node [above] {$\vec{f}_4$}; \draw (6.2,2.95) node {$w_1$};

\draw [->] (5.8, 4.6) -- (5.81, 4.6)node [below] {$\cev{f}_3$}; \draw [->] (5.3, 5.4) -- (5.3, 5.41)node [right] {$\vec{f}_5$}; \draw [->] (4.51, 4.9) -- (4.5, 4.9)node [above] {$\vec{f}_6$};\draw [->] (5, 4.21) -- (5, 4.2)node [left] {$\vec{f}_7$}; \draw (5.2,4.75) node {$w_2$};

\draw (3.2,4.75) node {$w_3$}; \draw [->] (3, 4.21) -- (3, 4.2)node [left] {$\cev{f}_4$}; \draw [->] (3.8, 4.6) -- (3.81, 4.6)node [below] {$\cev{f}_6$}; \draw [->] (3.3, 5.4) -- (3.3, 5.41)node [right] {$\vec{f}_8$}; \draw [->] (2.51, 4.9) -- (2.5, 4.9)node [above] {$\vec{f}_9$};

\draw (1.87,3.75) node {$w_4$}; \draw [->] (2, 4.3) -- (2, 4.31)node [right] {$\cev{f}_9$}; \draw [->] (1.31, 3.9) -- (1.3, 3.9)node [above] {$\cev{f}_5$}; \draw [->] (1.7, 3.21) -- (1.7, 3.2)node [left] {$\vec{f}_{10}$}; \draw [->] (2.4, 3.6) -- (2.41, 3.6)node [below] {$\cev{f}_7$};

\draw (4.17,1.15) node {$w_5$}; \draw [->] (4.3, 1.7) -- (4.3, 1.71)node [right] {$\cev{f}_2$}; \draw [->] (3.51, 1.3) -- (3.5, 1.3)node [above] {$\cev{f}_{10}$}; \draw [->] (4, 0.61) -- (4, 0.6)node [left] {$\cev{f}_8$}; \draw [->] (4.7, 1) -- (4.71, 1)node [below] {$\cev{f}_1$};

\end{tikzpicture}
\end{center}
\caption{$\Gamma_0$.}\label{Fig:exa2}
\end{figure}
It is easy to see that $\Gamma_0$ has one boundary component and $3$ standard cycles of lengths $5,3$ and $2$. Euler's equation implies the genus of $\Gamma_0$ is $3$. Therefore, the set of standard cycles of $\Gamma_0$ gives a minimal filling triple of $F_3$.
\end{exa}

%The following example has been discussed in \cite{BS}, which shows the existence of a minimal filling triple in $F_2$.
The following result shows the existence of a minimal filling $F_g(g\geq 2)$ of size $2g-1.$
\begin{lemma}\label{prop:girth}
For every integer $g\geq 3$, there exists a minimal filling of $F_g$ of size $2g-1$. 
\end{lemma}
\begin{proof}
To proof this lemma, we construct a $4$-regular fat graph $\Gamma$ of genus $g$ with a single boundary component and $(2g-1)$ standard cycles. Define $\Gamma = \left( E, \sim, \sigma_1, \sigma_0 \right) $  as follows:

\begin{enumerate}
\item $E=\left\lbrace \vec{e}_i, \cev{e}_i\;|\;i=1,2,\dots, 2(2g-1)\right\rbrace$.
\item The set of equivalence classes of $\sim$ is $V = \left\lbrace v_1, v_2, \dots, v_{2g-1}\right\rbrace,$ where 
\[
  v_j = \left\{\def\arraystretch{1.2}%
  \begin{array}{@{}c@{\quad}l@{}}
    \left( \vec{e}_1, \vec{e}_2, \vec{e}_3, \cev{e}_2 \right) & \text{if $j=1$},\\
     \left( \cev{e}_3, \vec{e}_5, \vec{e}_4, \vec{e}_6\right) & \text{if $j=2$},\\
    \left( \cev{e}_{2j}, \vec{e}_{2j+2}, \cev{e}_{2j-1}, \vec{e}_{2j+1}\right) & \text{if $3\leq j \leq 2g-2$}\text{ and}\\
    \left( \cev{e}_{4g-3}, \cev{e}_{4}, \cev{e}_{4g-2}, \cev{e}_{1}\right) & \text{if $j=2g-1$}.
  \end{array}\right.
\]
\item The fixed point free involution $\sigma_1$ is defined by $\sigma_1 \left( \vec{e}_i \right) = \cev{e}_i$ for $i=1,2,\dots, 2g-1.$
\item The permutation $\sigma_0$ is given by $\sigma_0=\prod\limits_{i=1}^{2g-1}v_i$.
\end{enumerate}
Observe that $\Gamma$ is a $4$-regular fat graph of genus $g$ with one boundary component  and $(2g-1)$ standard cycles, thus the lemma follows.
\end{proof}

To give the construction of a filling triple for the generic case, we follow the approach similar to as in~\cite{BS}, by making a connected sum construction with a fat graph on $F_2$ with four boundary components and two standard cycles, described in the following example.

\begin{exa}\label{eg:pair_2}
Consider the fat graph $G_2$ given in Figure~\ref{exa3}. The graph $G_2$ has two standard cycles and four boundary components. Euler's equation implies the genus of $G_2$ is $2$.
%\begin{enumerate}
%\item $E=\{\vec{x}_i, \vec{y}_i,\cev{x}_i, \cev{y}_i\;|\;i=1,2,\dots, 6\}$.
%
%\item The set of equivalence classes of $\sim$ is $V=\{v_1, v_2, \dots, v_6\}$, where 
%$ v_1=(\vec{x}_1,\cev{y}_6,\cev{x}_6,\vec{y}_1),\,\\v_2=(\cev{x}_2,\cev{y}_1,\vec{x}_3,\vec{y}_2),\,v_3=(\cev{x}_1,\cev{y}_2,\vec{x}_2,\vec{y}_3),v_4=(\vec{x}_4,\cev{y}_3,\cev{x}_3,\vec{y}_4),\,v_5=(\cev{x}_5,\cev{y}_4,\vec{x}_6,\vec{y}_5),\\v_6=(\cev{x}_4,\cev{y}_5,\vec{x}_5,\vec{y}_6).$
%\item The fixed point free involution $\sigma_1$ is defined by $\sigma_1(\vec{x}_i)=\cev{x}_i,\sigma_1(\vec{y}_i)=\cev{y}_i$, for $i=1,2,\dots, 6.$
%\item The permutation $\sigma_0$ is given by $\small \displaystyle \sigma_0=\prod_{i=1}^{6}v_i$.
%\end{enumerate}
\begin{figure}[htbp]
\begin{center}
\begin{tikzpicture}
\draw [rounded corners=2mm](2.3, 1.5) -- (2.3, 0) -- (7,0) -- (7, 3.8) -- (5.9, 3.8) -- (5.9, 5.3) -- (4.6, 5.3) -- (4.6, 6.8) -- (0,6.8) -- (0, 3) -- (1,3) -- (1, 1.5) -- cycle;

\draw [rounded corners=2mm] (3.3, 1.5)-- (3.3, 0.5) -- (4.6, 0.5) -- (4.6, 1.5) -- (5.9, 1.5) -- (5.9, 3.5) -- (6.7,3.5) -- (6.7, 0.3) -- (2.6,0.3) -- (2.6, 1.5) -- (4.3,1.5) -- (4.3, 0.8) -- (3.6, 0.8)--(3.6, 1.5);

\draw [rounded corners=2mm](3.3, 1.8) -- (3.3, 2.5) -- (2.6,2.5) -- (2.6, 1.8) -- (4.3, 1.8) -- (4.3, 3) -- (1.3, 3) -- (1.3, 1.8) -- (2.3, 1.8) -- (2.3, 2.8) -- (3.6, 2.8) -- (3.6, 1.8);

\draw [rounded corners=2mm](3.6, 5) -- (3.6, 4.3) -- (4.3,4.3) -- (4.3, 5) -- (2.6, 5) -- (2.6, 3.8) -- (5.6, 3.8) -- (5.6, 5) -- (4.6, 5) -- (4.6, 4) -- (3.3,4) -- (3.3, 5);

\draw [rounded corners=2mm](3.3, 5.3) -- (3.3, 6) -- (2.6,6) -- (2.6, 5.3) -- (4.3, 5.3) -- (4.3, 6.5) -- (0.3, 6.5) -- (0.3, 3.3) -- (1,3.3) -- (1, 5.3) -- (2.3,5.3) -- (2.3, 6.3) -- (3.6, 6.3 ) --(3.6, 5.3);

\draw [rounded corners=2mm](2.3, 5) -- (2.3, 3.5) -- (5.6,3.5) -- (5.6, 1.8) -- (4.6, 1.8) -- (4.6, 3.3) -- (1.3, 3.3)-- (1.3, 5) --cycle; 

\draw [->] (5.9, 4.3) -- (5.9, 4.31)node [right] {$\vec{y}_1$}; \draw [->] (3.91, 5.3) -- (3.9, 5.3)node [above] {$\vec{y}_2$}; \draw [->] (1.81, 5.3) -- (1.8, 5.3)node [above] {$\vec{y}_3$}; \draw [->] (1, 2.51) -- (1, 2.5)node [left] {$\vec{y}_4$}; \draw [->] (3, 1.5) -- (3.01, 1.5)node [below] {$\vec{y}_5$}; \draw [->] (5.1, 1.5) -- (5.11, 1.5)node [below] {$\vec{y}_6$}; \draw [->] (5.6, 3) -- (5.6, 3.01)node [left] {$\cev{y}_6$}; \draw [->] (3.9, 1.8) -- (3.89, 1.8)node [above] {$\cev{y}_5$}; \draw [->] (1.81, 1.8) -- (1.8, 1.8)node [above] {$\cev{y}_4$}; \draw [->] (1.3, 3.8) -- (1.3, 3.81)node [right] {$\cev{y}_3$}; \draw [->] (3, 5) -- (3.01, 5)node [below] {$\cev{y}_2$}; \draw [->] (5, 5) -- (5.01, 5)node [below] {$\cev{y}_1$};

\draw [->] (5.11,3.8) -- (5.1, 3.8)node [above] {$\vec{x}_1$}; \draw [->] (2.6,5.7) -- (2.6, 5.71)node [right] {$\vec{x}_2$}; \draw [->] (4.6,5.7) -- (4.6, 5.71)node [right] {$\vec{x}_3$}; \draw [->] (1.8, 3) -- (1.81, 3)node [below] {$\vec{x}_4$}; \draw [->] (4.3,1.11) -- (4.3, 1.1)node [left] {$\vec{x}_5$}; \draw [->] (2.3,1.11) -- (2.3, 1.1)node [left] {$\vec{x}_6$}; \draw [->] (6.3,3.5) -- (6.31, 3.5)node [below] {$\cev{x}_6$}; \draw [->] (2.6,2.2) -- (2.6, 2.21)node [right] {$\cev{x}_5$};  \draw [->] (4.6,2.2) -- (4.6, 2.21)node [right] {$\cev{x}_4$};  \draw [->] (0.61,3.3) -- (0.6, 3.3)node [above] {$\cev{x}_3$};  \draw [->] (4.3,4.61) -- (4.3, 4.6)node [left] {$\cev{x}_2$}; \draw [->] (2.3,4.61) -- (2.3, 4.6)node [left] {$\cev{x}_1$};

\draw (5.75, 3.65) node {$v_1$}; \draw (4.45, 5.15) node {$v_2$}; \draw (2.45, 5.15) node {$v_3$}; \draw (4.45, 1.65) node {$v_6$}; \draw (2.45, 1.65) node {$v_5$}; \draw (1.15, 3.15) node {$v_4$}; 
\end{tikzpicture}
\end{center}
\caption{$G_2$}\label{exa3}
\end{figure}
\end{exa}
\noindent Before going in to the proof of Proposition~\ref{prop:minimal}, we develop the following lemmas:
\begin{lemma}\label{thm:triple}
For every positive integer $g\geq 2$, there exists a minimal filling triple of $F_g$.
\end{lemma}

\begin{proof}
First, we consider the case, when $g$ is odd, i.e, $g=2k+1,k\geq 1$. We prove the lemma by induction on $k$. For $k=1$, the lemma holds by considering the fat graph $\Gamma_0$ described in Example~\ref{eg:triple_3}. Assume by the induction hypothesis that the lemma holds for all positive integers $k\leq m$, where $m\geq 1$. Then there exists a $4$-regular fat graph $\Gamma_{m}$ associated with a minimal filling triple on $F_g$, where $g=2m+1$. We take the fat graph $G_2$ described in Example~\ref{eg:pair_2} and we consider $\Gamma_{m+1}=\Gamma_m~\#_{(w_1,v_1)}G_2$, where $w_1$ is a vertex in $\Gamma_m$ at which there are no loops. Then using Proposition~\ref{connected_sum},(1)-(3), it follows that $\Gamma_{m+1}$ is a $4$-regular fat graph associated with a minimal filling triple on $F_{g+2}$ and the lemma follows by induction. A similar argument works for the case, when $g$ is even by taking $G_1$ (Example~\ref{eg:triple_2}) instead of $\Gamma_0$.
\end{proof}

\begin{lemma}\label{lem:plumb}
There exists a minimal filling of $F_{g+1}$ of size $s+2$ if $F_g$ has a minimal filling of size $s$.
\end{lemma}  

\begin{proof}
Suppose that $\Omega$ is a minimal filling of $F_g$ of size $s$. Then it corresponds to a $4$-regular fat graph $\Gamma$ of genus $g$ with one boundary component and $s$ standard cycles. To prove the lemma, we use the plumbing of $\Gamma$ with another fat graph $\Gamma_1$ of genus one with one boundary component and two standard cycles, as constructed in  Proposition~\ref{prop:minimal_filling}, and thus obtain a new fat graph ${\tilde \Gamma}=\Gamma\#_{(e,x)}\Gamma'$, where $e \text{ and } x$ are undirected edges of $\Gamma$ and $\Gamma_1$ respectively. By Proposition~\ref{plumb:boundary}, we see that $\tilde{\Gamma}$ is a $4$-regular fat graph of genus $g+1$ with one boundary component and $s+2$ standard cycles, and so the lemma follows. 

%Take $\Gamma'=(E,\sim,\sigma_1,\sigma_0)$  on $F_1$  given as follows:
%\begin{enumerate}
%\item $E=\{\vec{x}, \vec{y},\cev{x}, \cev{y}\}$.
%\item The set of equivalence classes of $\sim$ is $V=\{v\}$ where 
%$ v=\{\vec{x},\cev{y},\cev{x},\vec{y}\}.$
%\item The fixed point free involution $\sigma_1$ is defined by $\sigma_1(\vec{x})=\cev{x},\sigma_1(\vec{y})=\cev{y}$ .
%\item The permutation $\sigma_0$ is given by $(\vec{x},\cev{y},\cev{x},\vec{y})$.
%\end{enumerate}
 
\end{proof}

\begin{lemma}\label{lem:genus_3}
For  each $2\leq s\leq 6$, there exists a minimal filling of $F_3$ of size $s$.
\end{lemma}
\begin{proof}
The case, when $s=2$ follows from Theorem 1.4~\cite{BS} and the case, when $s=6$, follows from Theorem~\ref{res:1}. It follows from Example~\ref{eg:triple_3} and the fat graph constructed in Lemma~\ref{prop:girth} that the result is true for $s=3,5$. It is enough to prove that there exists a minimal filling quadruple for $F_3$. We prove the result by constructing a 4-regular fat graph $\Gamma=(E,\sim,\sigma_1,\sigma_0)$  of genus $3$, having one boundary component and $4$ standard cycles,  described as follows:

\begin{enumerate}
\item $E=\{\vec{f}_i, \cev{f}_i\;|\;i=1,2,\dots, 10\}$.
\item The set of equivalence classes of $\sim$ is $V=\{w_1, w_2, \dots, w_5\}$, where 
$ w_1=(\vec{f}_1, \vec{f}_2, \vec{f}_3, \cev{f}_2),\,\\w_2=(\cev{f}_3, \vec{f}_4, \vec{f}_5, \vec{f}_6),\,w_3=(\cev{f}_6, \cev{f}_9, \cev{f}_{10}, \cev{f}_1),\,w_4=(\cev{f}_5, \vec{f}_{7}, \vec{f}_8, \cev{f}_7) \text{ and }w_5=(\cev{f}_4, \vec{f}_9, \vec{f}_{10}, \cev{f}_8).$
\item The fixed point free involution $\sigma_1$ is defined by $\sigma_1(\vec{f}_i)=\cev{f}_i$ for $i=1,2,\dots, 10.$
\item The permutation $\sigma_0$ is given by $\sigma_0=\prod\limits_{i=1}^{5}w_i$.
\end{enumerate}
\end{proof}
\noindent We are now ready to prove Proposition~\ref{prop:minimal}.
\begin{proof}[Proof of Proposition~\ref{prop:minimal}]
 We prove the result by induction on $g$. It follows from Theorem 1.2 in~\cite{BS} and Theorem~\ref{res:1} that the result holds true for $g=2$ .  By Lemma~\ref{lem:genus_3}, the result is true for $g=3$. Assume that the result holds for all
$g\leq m$, where $m\geq 3$.  We prove the result for $g=m+1$. By induction hypothesis, for each $2\leq s\leq 2m$, there exists a fat graph $\Gamma_{m,s}$ of genus $m$ with one boundary component and $s$ standard cycles. By Lemma~\ref{lem:plumb}, for each $2\leq s \leq 2m$, there exists a fat graph ${\tilde \Gamma}_{m+1,s+2}$ which corresponds to a minimal filling of $F_{m+1}$ of size $s+2$, i.e., for $4\leq s'\leq 2m+2$, there exists a minimal filling of $F_{m+1}$ of size $s'$. To complete the proof, we have to show that there exists a minimal filling pair and triple for $F_{m+1}$, which follows from Theorem 1.4 in~\cite{BS} and Lemma~\ref{thm:triple} respectively.
The proof now follows by induction. 
\end{proof}
%
%\begin{theorem}
%Let $g\geq 3$ be a positive integer. Then for each $2\leq s\leq 2g$, there exists a minimal filling set $S$ of $F_g$ of size $s$. 
%\end{theorem}
%
%\begin{proof}
%We prove the result by induction on $g$. By Lemma~\ref{lem:genus_3}, the result is true for $g=3$. Assume that the result holds for all
%$g\leq m$.  We prove the result for $g=m+1$. By induction hypothesis, for each $2\leq s\leq 2m$, there exists a fat graph $\Gamma_{m,s}$ such that $g(\Gamma_{m,s})=m,b(\Gamma_{m,s})=1$ and $\#S(\Gamma_{m,s})=s$. By Lemma~\ref{lem:plumb}, for each $2\leq s \leq 2m$, there exists a fat graph ${\tilde \Gamma}_{m+1,s+2}$ such that it corresponds to a minimal filling of $F_{m+1}$ of size $s+2$ i.e, for $4\leq s'\leq 2m+2$, there exists a minimal filling of $F_{m+1}$ of size $s'$. To complete the proof, we have to show that there exists a minimal filling pair and triple for $F_{m+1}$, which follows from Theorem~\ref{thm:triple} and the work of Sanki~\cite{BS}.
%The proof now follows by induction. 
%\end{proof}
 
%%%%%%%%%%%%%%%%%%%%% Section 5: Filling systems of all possible sizes %%%%%%%%%%%%%%%%%%%%%%%
\section{Filling systems of all possible sizes}\label{sec:5}
In this section, we prove Theorem~\ref{res:2}. Before going to the proof, we consider the following proposition which is essential for the proof the the theorem.

\begin{prop}\label{prop:5.1}
For every $g\geq 2$ and $b\geq 1$ with $(g, b)\neq (2,1)$, there exists a 4-regular fat graph $\Gamma(g,b)$ of genus $g$ and $b$ boundary components satisfying folloing:
\begin{enumerate}
  \item The number of standard cycles in $\Gamma(g, b)$ is $2$.
  \item There is an edge $x=\{\vec{x}, \cev{x}\}$ and a boundary component $\partial$ of $\Gamma(g, b)$ such that the directed edges $\vec{x}$ and $\cev{x}$ are both in $\partial$.  
\end{enumerate} 
\end{prop}
\begin{proof}
For all $g\geq 3$ and $b\geq 1$, we consider $\Gamma(g, b) := G^g_b$, where $G^g_b$ is the $4$-regular fat grpah of genus $g$ with $b$ boundary components and  two standard cycles, explicitly constructed in the proof of Theorem 1.4 in~\cite{BS} (see Section 5 in~\cite{BS}). Furthermore, in the case, when $g=2$, we consider the $4$-regular decorated fat graph $\Gamma(2,b)=\left( E, \sim, \sigma_1, \sigma_0 \right)$ described as follows (see Figure~\ref{fig:sec-5}).
%define $\Gamma(g, b):=G^2_b$, where $G^2_b$ is explicitly given in the proof of Lemma 5.1, Case 2, in~\cite{BS}. It is easy to see that these graphs satisfy Proposition~\ref{prop:5.1}. Therefore, now we consider $g=2$ and $b$ is an even integer.

\begin{enumerate}
\item $E=\left\lbrace \vec{e}_i, \cev{e}_i, \vec{f}_i,\cev{f}_i\;|\;i=1,2,\dots,  b+2  \right\rbrace$.
\item The set of equivalence classes of $\sim$ is $V=\left\lbrace v_j|\; j=1, \dots, b+2 \right\rbrace$, where 
\[
  v_j = \left\{\def\arraystretch{1.2}%
  \begin{array}{@{}c@{\quad}l@{}}
 \left( \vec{e}_1, \vec{f}_1, \cev{e}_{b+2}, \cev{f}_{b+2} \right) & \text{if $j=1$},\\
    \left( \cev{e}_{j-1}, \cev{f}_{j-1}, \vec{e}_j,\vec{f}_j\right) & \text{if $2\leq j \leq b$},\\
    \left(\cev{e}_{b},\vec{f}_{b+2},\vec{e}_{b+1},\cev{f}_{b+1}\right) & \text{if $j=b+1$} \text{ and}\\
    \left( \cev{e}_{b+1},\vec{f}_{b+1},\vec{e}_{b+2},\cev{f}_{b} \right) & \text{if $j=b+2$}.\\
 
  \end{array}\right.
\]
\item The fixed point free involution $\sigma_1$ is defined by $\sigma_1 \left( \vec{e}_i \right) = \cev{e}_i$ for $i=1,2,\dots,   b+2.$
\item The permutation $\sigma_0$ is given by $\sigma_0=\prod\limits_{j=1}^{b+2}v_j$.

\end{enumerate}

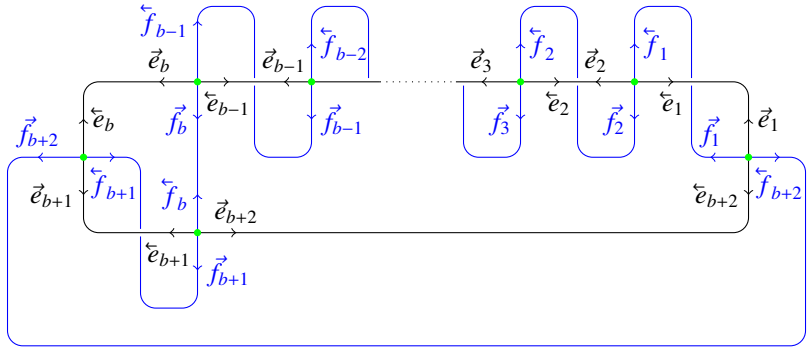
\begin{figure}[htbp]
\begin{center}
\begin{tikzpicture}
\draw [rounded corners = 2mm] (4.9, 3.5) -- (1, 3.5) -- (1,1.5) -- (9.75, 1.5) -- (9.75, 3.5) -- (5.9, 3.5);
\draw [dotted] (4.9, 3.5) -- (5.9, 3.5);
\draw [rounded corners=2mm, blue] (1.75, 1.55) -- (1.75, 2.5) -- (0, 2.5) -- (0, 0) -- (10.5, 0) -- (10.5, 2.5) -- (9, 2.5) -- (9, 3.45);  \draw [rounded corners=2mm, blue] (9, 3.55) -- (9, 4.5) -- (8.25,4.5) -- (8.25, 2.5) -- (7.5, 2.5) -- (7.5, 3.45); \draw [rounded corners=2mm, blue] (7.5, 3.55) -- (7.5, 4.5) -- (6.75, 4.5) -- (6.75, 2.5) -- (6, 2.5) -- (6, 3.45); \draw [rounded corners=2mm, blue] (4.75, 3.55) -- (4.75, 4.5) -- (4, 4.5) -- (4, 2.5) -- (3.25, 2.5) -- (3.25, 3.45); 
\draw [rounded corners=2mm, blue] (3.25, 3.55) -- (3.25, 4.5) -- (2.5, 4.5) -- (2.5, 0.5) -- (1.75, 0.5) -- (1.75, 1.45); 
\draw [green, fill] (1,2.5) circle [radius =0.4mm]; \draw [green, fill] (9.75,2.5) circle [radius =0.4mm]; \draw [green, fill] (2.5, 1.5) circle [radius =0.4mm]; \draw [green, fill] (2.5, 3.5) circle [radius =0.4mm]; \draw [green, fill] (4, 3.5) circle [radius =0.4mm]; \draw [green, fill] (6.75, 3.5) circle [radius =0.4mm]; \draw [green, fill] (8.25, 3.5) circle [radius =0.4mm];

\draw [->] (9.75, 3) -- (9.75, 3.01)node [right] {$\vec{e}_1$}; \draw [->] (9.75, 2.01) -- (9.75, 2)node [left] {$\cev{e}_{b+2}$}; \draw [->] (7.75, 3.5) -- (7.74, 3.5)node [above] {$\vec{e}_2$}; \draw [->] (6.25, 3.5) -- (6.24, 3.5)node [above] {$\vec{e}_3$}; \draw [->] (8.75, 3.5) -- (8.76, 3.5)node [below] {$\cev{e}_1$}; \draw [->] (7.25, 3.5) -- (7.26, 3.5)node [below] {$\cev{e}_2$};
\draw [->] (3.65, 3.5) -- (3.64, 3.5)node [above] {$\vec{e}_{b-1}$}; \draw [->] (2.9, 3.5) -- (2.91, 3.5)node [below] {$\cev{e}_{b-1}$}; \draw [->] (2.01, 3.5) -- (2, 3.5)node [above] {$\vec{e}_{b}$};
\draw [->] (3, 1.5) -- (3.01,1.5)node [above] {$\vec{e}_{b+2}$}; \draw [->] (2.14, 1.5) -- (2.13,1.5)node [below] {$\cev{e}_{b+1}$}; \draw [->] (1, 2.01) -- (1, 2)node [left] {$\vec{e}_{b+1}$}; \draw [->] (1,3) -- (1,3.01)node [right] {$\cev{e}_{b}$};

\draw [->, blue] (9.26,2.5) -- (9.25, 2.5)node [above] {$\vec{f}_1$}; \draw [->, blue] (10.15,2.5) -- (10.16, 2.5)node [below] {$\cev{f}_{b+2}$}; \draw [->, blue](8.25, 3.01) -- (8.25,3)node [left] {$\vec{f}_2$}; \draw [->, blue](8.25, 4) -- (8.25,4.01)node [right] {$\cev{f}_1$}; \draw [->, blue](6.75, 3.01) -- (6.75,3)node [left] {$\vec{f}_3$}; \draw [->, blue](6.75, 4) -- (6.75,4.01)node [right] {$\cev{f}_2$}; \draw [->, blue](4, 3.01) -- (4,3)node [right] {$\vec{f}_{b-1}$}; \draw [->, blue](4, 4) -- (4,4.01)node [right] {$\cev{f}_{b-2}$}; \draw [->, blue](2.5, 3.01) -- (2.5,3)node [left] {$\vec{f}_{b}$}; \draw [->, blue](2.5, 4.3) -- (2.5,4.31)node [left] {$\cev{f}_{b-1}$}; \draw [->, blue](2.5, 2) -- (2.5,2.01)node [left] {$\cev{f}_{b}$}; \draw [->, blue](2.5, 1.01) -- (2.5,1)node [right] {$\vec{f}_{b+1}$}; \draw [->, blue](1.4, 2.5) -- (1.41,2.5)node [below] {$\cev{f}_{b+1}$}; \draw [->, blue](0.41, 2.5) -- (0.4,2.5)node [above] {$\vec{f}_{b+2}$};
\end{tikzpicture}
\end{center}
\caption{Fat graph for filling pair of a genus 2 surface with $b$ complementary components.}\label{fig:sec-5}
\end{figure}

Now, the boundary components of the graph $\Gamma(2,b)$ are given by :
\begin{align*}
  \partial_1=&\vec{e}_1\cev{f}_1\cev{e}_{b+2}\cev{f}_b\cev{e}_{b-1}\vec{f}_{b-1}\vec{e}_b\vec{f}_{b+2},  \\
  \partial_j=& \vec{e}_j\cev{f}_j\cev{e}_{j-1}\vec{f}_{j-1},\text{ for }2\leq j\leq b-1,\text{ and}\\
    \partial_b=& \cev{f}_{b+2}\vec{e}_{b+1}\vec{f}_{b+1}\cev{e}_b\vec{f}_b\cev{e}_{b+1}\cev{f}_{b+1}\vec{e}_{b+2}.
\end{align*}

It is apparant that $\vec{f}_{b+1},\cev{f}_{b+1}\in \partial_b$, and $\Gamma(2,b)$ has two standard cycles. This completes the proof.
\end{proof}

\begin{exa}\label{eg:sec5}
Consider the graph given in Figure~\ref{fig:sec5}. Euler's equation implies that the genus of the graph is $2$. It has two boundary components given by, 
\begin{align*}
  \partial_1=&\vec{x}_1 \vec{y}_2 \cev{z}_1 \vec{x}_3 \vec{y}_1 \cev{x}_1 \cev{y}_3 \vec{z}_2 \cev{x}_2 \cev{y}_1 \text{ and} \\
  \partial_2=& \vec{x}_2 \vec{z}_1\vec{y}_3\cev{x}_3\cev{z}_2\cev{y}_2.
\end{align*}
 It is straightforward to see that the graph  has three standard cycles (see Figure~\ref{fig:sec5}).
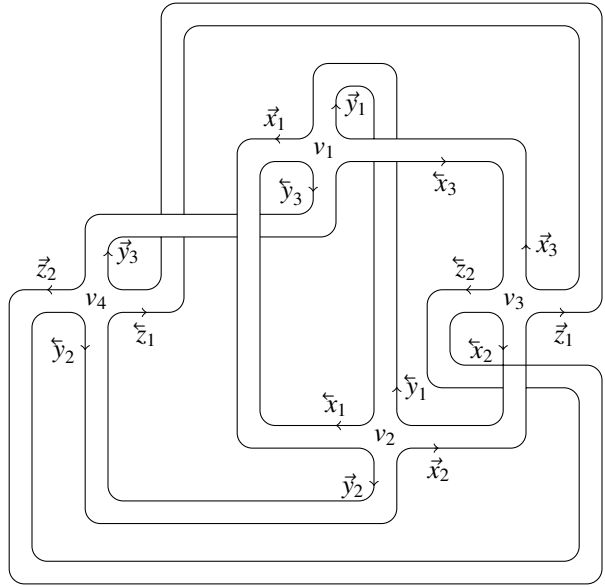
\begin{figure}[htbp]
\begin{center}
\begin{tikzpicture}
\draw [rounded corners=2mm] (6.8,2.6) -- (7.5,2.6) --(7.5,0.3) -- (0.3,0.3) -- (0.3,3.6) -- (1,3.6) -- (1,0.8)-- (5.1,0.8) -- (5.1,1.8) -- (6.8,1.8) -- (6.8,3.6) -- (7.8,3.6) -- (7.8,7.7) -- (2,7.7) -- (2,4.9);

\draw [rounded corners=2mm] (6.8,2.9) -- (7.8,2.9) -- (7.8,0) -- (0, 0) -- (0,3.9) -- (1,3.9) -- (1,4.9) -- (3,4.9); 
\draw [rounded corners=2mm] (3,4.6) -- (1.3,4.6) -- (1.3,3.9) -- (2,3.9) -- (2,4.6); 
\draw [rounded corners=2mm] (2.3,4.6) -- (2.3,3.6) -- (1.3,3.6) -- (1.3,1.1) -- (4.8,1.1) -- (4.8,1.8) -- (3,1.8) -- (3,5.9) -- (4,5.9) -- (4,6.9) -- (5.1,6.9) -- (5.1,5.9); 
\draw [rounded corners=2mm] (4.8,5.9) -- (4.8,6.6) -- (4.3,6.6) -- (4.3,5.9) -- (6.8,5.9) -- (6.8,3.9) -- (7.5,3.9) -- (7.5,7.4) -- (2.3,7.4) -- (2.3,4.9); 
\draw [rounded corners=2mm] (6.5,2.6) -- (5.5,2.6) -- (5.5,3.9) -- (6.5,3.9) -- (6.5,5.6) -- (4.3,5.6) -- (4.3,4.6) -- (3.3,4.6); 
\draw [rounded corners=2mm] (3.3,4.9) -- (4,4.9) -- (4,5.6)-- (3.3,5.6) -- (3.3,2.1) -- (4.8,2.1) -- (4.8,5.6); \draw [rounded corners=2mm] (5.1,5.6) -- (5.1,2.1) -- (6.5,2.1) -- (6.5,3.6) -- (5.8,3.6) -- (5.8,2.9) -- (6.8,2.9);

\draw (4.15, 5.75)node {$v_1$}; \draw (4.95, 1.95)node {$v_2$}; \draw (6.65, 3.75)node {$v_3$}; \draw (1.15, 3.75)node {$v_4$};

\draw [->] (3.55, 5.9) -- (3.5, 5.9)node [above] {$\vec{x}_1$}; \draw [->] (5.6, 1.8) -- (5.65, 1.8)node [below] {$\vec{x}_2$}; \draw [->] (6.8, 4.45) -- (6.8, 4.5)node [right] {$\vec{x}_3$}; \draw [->] (5.7, 5.6) -- (5.75, 5.6)node [below] {$\cev{x}_3$}; \draw [->] (6.5, 3.15) -- (6.5, 3.1)node [left] {$\cev{x}_2$}; \draw [->] (4.31, 2.1) -- (4.3, 2.1)node [above] {$\cev{x}_1$};

\draw [->] (7.3, 3.6) -- (7.31, 3.6)node [below] {$\vec{z}_1$};\draw [->] (0.51, 3.9) -- (0.5, 3.9)node [above] {$\vec{z}_2$}; \draw [->] (6.02, 3.9) -- (6.01, 3.9)node [above] {$\cev{z}_2$};\draw [->] (1.8, 3.6) -- (1.81, 3.6)node [below] {$\cev{z}_1$};

\draw [->] (4.3, 6.35) -- (4.3, 6.36)node [right] {$\vec{y}_1$}; \draw [->] (4, 5.21) -- (4, 5.2)node [left] {$\cev{y}_3$}; \draw [->] (4.8, 1.31) -- (4.8, 1.3)node [left] {$\vec{y}_2$}; \draw [->] (5.1, 2.6) -- (5.1, 2.61)node [right] {$\cev{y}_1$}; \draw [->] (1.3, 4.4) -- (1.3, 4.41)node [right] {$\vec{y}_3$}; \draw [->] (1, 3.11) -- (1, 3.1)node [left] {$\cev{y}_2$};
\end{tikzpicture}
\end{center}
\caption{A fat graph of genus $2$ with three standard cycles and two boundary components.}\label{fig:sec5}
\end{figure}

\end{exa}

We are now ready to prove Theorem~\ref{res:2} for the generic case.
% \begin{proof}
 %In view of  Proposition~\ref{prop:minimal}, we need to prove the result for $b\geq 2$. We will prove the result by induction on $b$. Two cases needs to be considered accordingly as the genus $g$ equals two and the genus is greater than two. We first consider the case when $g>2$. Once again the Proposition~\ref{prop:minimal}  implies that for each $2\leq s\leq 2g$, there exists a minimal filling $\Gamma_1^s$ of size $s$.  

  %Consider $ G_2^{s+1} = \left( \Gamma_1^s, x \right) \# \left( \Gamma_1, y\right)$, where $x=\{\vec{x},\cev{x}\}$ and $y=\{\vec{y},\cev{y}\}$ are edges of $\Gamma_1^s$ and $\Gamma_1$ respectively. It follows from (1)-(3) Proposition~\ref{prop:join} that $G_2^{s+1}$ corresponds to a filling multicurve of size $s+1$ such that the complement is a pair of disjoint discs. Therefore, for  $3\leq t\leq 2g+1$,  we get a filling multicurve of $F_g$ of size $t$ with two boundary components. Then the result holds for $b=2$, by the existense of filling pair with complement a pair of disjoint discs. It is implicit in the proof of Proposition~\ref{prop:join} that there exists an edge $z$ of $G_2^{s+1}$ such that $\vec{z},\cev{z}\in  \eta$ for some $\partial G_2^{s+1}$.
 %\end{proof}

\begin{proof}[Proof of Theorem~\ref{res:2}]

Let us consider $g\geq 2$ and a pair of integers $(b, s)$, where $b\geq 1$ and $L_{g,b}\leq s \leq U_{g, b}$. We construct a 4-regular fat graph of genus $g$ with $b$ boundary components and $s$ standard cycles. There are the following two cases to be considered. 

\noindent {\bf Case 1 ($b\geq s$).} Let $\Gamma_2(g, b')=\Gamma(g, b')$ be the fat graph satisfying Proposition~\ref{prop:5.1}, where $b'=b-s+2\geq 2$. Now, for $i\geq 2$, we inductively define $$\Gamma_{i+1}(g, b'+i-1) := \left( \Gamma_i(g, b'+i-2), x \right)\#\left(\Gamma_1, y\right),$$ where $x={\vec{x}, \cev{x}}$ is an undirected edge of $\Gamma_i(g, b'+i-2)$ such that the directed edges $\vec{x}, \cev{x}$ are both in the same boundary component of $\Gamma_i(g, b'+i-2)$, and  $\Gamma_1$ as in the proof of Proposition~\ref{prop:minimal_filling}. We note that the existence of such an edge $x$ is ensured by Proposition~\ref{prop:5.1} when $i=2$, and the operation join of fat graphs, when $i\geq 3$. Now, by Proposition~\ref{prop:join}, observe that $\Gamma_i(g, b'+i-2)$ is a fat graph of genus $g$ with $(b'+i-2)$ boundary components and $i$ standard sycles. In particular, for $i=s$, we have $\Gamma_s(g, b'+s-2)=\Gamma_s(g, b)$, which has the desired properties.\\
{\bf Case 2 ($b< s$).} Let $G_1(g, s-b+1)$ be a $4$ regular fat graph of genus $g$ with a single boundary component and $s-b+1$ standard cycles, where $(g, b, s)\neq (2,1,2)$ (see Proposition~\ref{prop:minimal}). Furthermore, we take $G_2(2, 3)$ to be the graph given in example~\ref{eg:sec5} (see Figure~\ref{fig:sec5}). Now, for $i\geq 1$, we define $$G_{i+1}\left( g, s-b+i+1\right)= \left( G_i\left( g, s-b+i\right), x \right)\# (\Gamma_1, y).$$ 
Now, by Proposition~\ref{prop:join}, observe that $G_i(g, s-b+i)$ is a fat graph of genus $g$ with $i$ boundary components and $s-b+i$ standard sycles. In particular, for $i=b$, we have $G_b(g, s)$ is a required fat graph.

\end{proof}

%%%%%%%%%%%% Section 5: Proof of Theorem 1.3 %%%%%%%

\section{Geometric intersection number of minimal fillings}
In this section, we proof Theorem~\ref{res:3}. To this end, we define the notion of \emph{weighted intersection graph} and develop some results.

A weighted graph $G$ is a triple $G = \left( V, E, \omega \right)$, where $V$ is the set of vertices, $E$ is the set of (undirected) edges, and $\omega: E \to \mathbb{R}_+$ is a function which assigns each edge $e\in E$ to its weight $\omega \left( e \right)$. For $v\in V$, we define the degree of $v$ as $\Deg \left( v \right) = \sum_{e\in \Star \left( v \right) } \omega \left( e \right)$, where $\Star \left( v \right)$ is the set of all edges incident at $v$. Further, we define 
$$\omega_{\max} \left( G \right)=\max \left\lbrace \omega(e) | \; e\in E \right\rbrace.$$

\begin{dfn}
Given a decorated fat graph $\Gamma$, the weighted intersection graph, denoted by $W\left( \Gamma \right) = \left( V, E, \omega_{\Gamma} \right)$, is given by following:
\begin{enumerate}
\item The set of vertices $V$ is the set of standard cycles of $\Gamma$.
\item There is a simple edge between two vertices, if the corresponding standard cycles intersect and $E$ is the set of all edges.
\item If $e$ is an edge joining two vertices $C_i$ and $C_j$, then $\omega_{\Gamma} \left( e \right) = i \left( C_i, C_j \right).$
\end{enumerate}
\end{dfn}
\subsection{Proof of the inequality of Theorem~\ref{res:3}}
Let $\Gamma$ be a decorated fat graph associated with a minimal filling of $F_g$. To prove the inequality in Theorem~\ref{res:3}, it suffices to prove the proposition below.
\begin{prop}\label{lem:WIG}
Let $G=W\left( \Gamma \right)$ be the weighted intersection graph of a decorated fat graph $\Gamma$ of genus $g\geq 2$. Then we have $$\omega_{\max} \left( G \right) \leq 2g-s +1,$$ where $s$ is the number of standard cycles of $\Gamma$.
\end{prop}

\begin{proof}
For $s=2$, the two standard cycles have intersection number $2g-1$, which implies that $\omega_{\max} \left( G \right)=2g-1$. Thus equality holds in this case.

Assume that $s\geq 3$. There exists an edge $e_0$ such that  $\omega_{\max}\left( G \right) = \omega_{\Gamma}\left( e_0 \right).$ Let $C_0$ and $C_1$ be the vertices of the edge $e_0$. Thus the degree of each of these vertices is at least $\omega_{\max} \left( G \right)$. Further, $G$ being  connected implies that one of $C_0$ and $C_1$ has degree at least $\omega_{\max} \left( G \right) + 1$. So, we assume that $deg \left( C_0 \right) \geq \omega_{\max} \left( G \right) +1$ and $deg \left( C_1 \right) \geq \omega_{\max} \left( G \right).$ Furthermore, by the connectedness, we can index the vertices of the intersection graph as $C_i, i< s$ such that $deg \left( C_i \right) \geq 2$ for $3 \leq i \leq s-2$ and $deg \left( C_{s-1} \right)\geq 1.$  By the minimality condition, we have 
\begin{align*}
& 4g-2 = \sum\limits_{i=0}^{s-1} deg(v)\geq \left( \omega_{\max} \left( G \right) + 1 \right) + \omega_{\max} \left( G \right) + 2(s-3) +1 \\
%\Rightarrow & 4g-2 \geq 2 \omega_{\max} \left( G \right) + 2s  - 4 \\ 
\Rightarrow & \omega_{\max}\left( G \right) \leq 2g - s +1.
\end{align*}  
\end{proof}

\subsection{The tightness of the inequality in Theorem~\ref{res:3}}
In the following lemma, we prove the tightness of the inequality for the case, when the size of minimal filling is three. 

\begin{lemma}\label{weighted_triple}
For $g\geq 2$, there exists a $4$-regular fat graph $\Gamma$ of genus $g$ with one boundary component and $3$ standard cycles such that its weighted intersection graph $G=W\left(\Gamma \right)$ satisfies $$\omega_{\max}(G)=2g-2.$$
\end{lemma}

\begin{proof}
The case, when $g=2$ follows from Theorem 1.2~\cite{BS}. Now, we consider $g\geq 3$. It follows from Theorem 1.4~\cite{BS} that there exists a filling pair $(\alpha,\beta)$ of $F_{g-1}$ such that the complement $S_{g-1}\setminus \{\alpha,\beta\}$ 
is a disjoint union of two topological discs. Let $\Gamma_0$ be the associated fat graph with this filling. Then $\Gamma_0$ is a $4$-regular fat graph having $2g-2$ vertices 
and two boundary components.

Let $\partial\Gamma_0=\{\partial_1,\partial_2\}$. We choose an edge $x=\{ \vec{x},  \cev{x}\}$ in $\Gamma_0$, such that $\vec{x}\in \partial_1$ and $\cev{x}\in \partial_2$ (the existence of such an edge follows from Theorem 1.4~\cite{BS}). Let $\Gamma_0'$ be the fat graph on sphere with two boundary components and one standard 1-cycle. Take $\Gamma=\Gamma_0\#_{(x,y)}\Gamma_0'$, 
where $y=\{\vec{y},\cev{y}\}$ is the only edge of $\Gamma_0'$. It follows from Proposition~\ref{plumb:boundary} that $\lvert \partial\Gamma \rvert=1$ and thus a simple Euler characteristic argument implies that $\Gamma$ has genus $g$. It is easy to see that 
$\Gamma$ has the desired properties and the result follows.
\end{proof}
Now, the lemma below completes the proof of tightness of Theorem~\ref{res:3}.
\begin{lemma}\label{lem:weighted_even}
For every $g\geq 2$ and $s$ satisfying $L_{g,1} \leq s \leq U_{g,1}$, there exists a 4-regular fat graph $\Gamma$ of genus $g$ with $s$ standard cycles,  and one boundary component whose weighted intersection graph $G=W\left( \Gamma \right)$ satisfies $$\omega_{\max}(G)=2g-s+1.$$ 
\end{lemma}

\begin{proof}
We consider two cases, when $s$ is even and odd.

\noindent {\bf Case 1.} In this case, we consider $s$ is even. It follows from Theorem 1.2~\cite{BS} that the result is true for $g=2$. The construction of $\Omega_g^{\max}$ in the proof of Theorem~\ref{res:1} implies that the result holds for $g\geq 3$ when $s=U_{g,1}(=2g)$. %Further, the result follows from Theorem 1.4~\cite{BS} in the case when $s=2$.
 
Now, Lemma~\ref{weighted_triple} implies that the result is true for $g\geq 3$ when $s=2g-2$. We assume that $g\geq 3$ and prove the result by constructing a fat graph with the desired properties. The proof of Proposition~\ref{lem:WIG} implies the equality  in the case, when $s=2$. Therefore, we assume that $s=2m$, where $2\leq m\leq g-2$ and this implies $g':=g-m+1\geq 3$. So, there exists a fat graph $\Gamma_1$ of genus $g'$ with $2$ standard cycles and a single boundary component. Now, consider the fat graph $\Gamma_2$ of genus $m-1 \left( \geq 1 \right)$ with $2m-2$ standard cycles and a single boundary component whose intersection graph is a simple path of length $2m-3$. We define $\Gamma=\Gamma_1 \# \Gamma_2$, where the plumbing can be done along any two edges of $\Gamma_1$ and $\Gamma_2$. It follows from Proposition~\ref{plumb:boundary} that the number of boundary component in $\Gamma$ is one. Further, it is easy to see that the genus of $\Gamma$ is $g$ and it has $s$ standard cycles. Moreover, $\omega_{\max}(G)=2g-s+1$, which is realized by the edge in the weighted intersection graph joining the vertices corresponding the standard cycles in $\Gamma_1$. 

\noindent {\bf Case 2.} In this case, we consider $s$  is odd. 
% each positive integer $2\leq s\leq 6$, there exists a a 4-regular fat graph $\Gamma$ of genus $g$ with $s$ standard cycles,  and one boundary component whose intersection graph $G=W\left( \cap \left( \Gamma \right) \right)$ satisfies $\omega_{\max}(G)=2g-s+1$.
We prove result by induction on $g$. For $g=3$, the result is a direct consequence of Lemma~\ref{weighted_triple} and Lemma~\ref{prop:girth}. Suppose that the result is true for $g=m$, where $m\geq 3$. Now, consider $g=m+1$ and odd integer $s$ such that $3\leq s\leq 2m+1$. By Lemma~\ref{weighted_triple} and Lemma~\ref{prop:girth}, it is enough to consider $5\leq s\leq 2m-1$. Take $s=2k-1$ where $3\leq k \leq m$. By induction hypothesis, there exists a 4-regular fat graph $\Gamma'$ of genus $m$ with $2k-3$ standard cycles,  and one boundary component whose intersection graph $G'=W\left( \Gamma' \right)$ satisfies $\omega_{\max}(G')=2m-(2k-3)+1$. To prove the result, we use the
plumbing of $\Gamma'$ with another fat graph $\Gamma_0$ of genus one with one boundary component and two standard cycles and hence obtain a fat graph $\Gamma$ of genus $g=m+1$ with $2k-1$ 
standard cycles and one boundary component whose weighted intersection graph $G=W\left( \Gamma \right)$ satisfies 
$$\omega_{\max}(G)=\omega_{\max}(G')=2m-(2k-3)+1= 2(m+1)-(2k-1)+1.$$ Thus, the result follows by induction.
\end{proof}

\bibliographystyle{alpha}
\bibliography{fillings}

\begin{thebibliography}{APP11}

\bibitem[AH15]{TA}
Tarik Aougab and Shinnyih Huang.
\newblock Minimally intersecting filling pairs on surfaces.
\newblock {\em Algebr. Geom. Topol.}, 15(2):903--932, 2015.

\bibitem[APP11]{JA}
James~W. Anderson, Hugo Parlier, and Alexandra Pettet.
\newblock Small filling sets of curves on a surface.
\newblock {\em Topology Appl.}, 158(1):84--92, 2011.

\bibitem[FM12]{FM}
Benson Farb and Dan Margalit.
\newblock {\em A primer on mapping class groups}, volume~49 of {\em Princeton
  Mathematical Series}.
\newblock Princeton University Press, Princeton, NJ, 2012.

\bibitem[FP16]{Parlier}
Federica Fanoni and Hugo Parlier.
\newblock Filling sets of curves on punctured surfaces.
\newblock {\em New York J. Math.}, 22:653--666, 2016.

\bibitem[Pen88]{RCP}
Robert~C. Penner.
\newblock A construction of pseudo-{A}nosov homeomorphisms.
\newblock {\em Trans. Amer. Math. Soc.}, 310(1):179--197, 1988.

\bibitem[San17]{BS}
Bidyut Sanki.
\newblock Filling of closed surfaces.
\newblock {\em J. Topol. Anal.}, doi:10.1142/S1793525318500309, 2017.

\bibitem[SS99]{Schmutz}
Paul Schmutz~Schaller.
\newblock Systoles and topological {M}orse functions for {R}iemann surfaces.
\newblock {\em J. Differential Geom.}, 52(3):407--452, 1999.

\bibitem[Thu86]{WT}
William Thurston.
\newblock A spine for {T}eichm{\"u}ller space.
\newblock {\em preprint}, 1986.

\bibitem[Thu88]{TW2}
William~P. Thurston.
\newblock On the geometry and dynamics of diffeomorphisms of surfaces.
\newblock {\em Bull. Amer. Math. Soc. (N.S.)}, 19(2):417--431, 1988.

\end{thebibliography}
\end{document}